\documentclass[3p,10pt,nopreprintline]{elsarticle}

\usepackage[T1]{fontenc}
\usepackage{lmodern}
\usepackage{amsmath,amssymb,amsthm,mathtools,mathrsfs}
\usepackage{booktabs,tabularx,array}
\usepackage{microtype}
\usepackage{enumitem}
\usepackage{needspace}
\usepackage{placeins}
\usepackage[hidelinks]{hyperref}
\usepackage[nameinlink,noabbrev]{cleveref}
\hypersetup{
  pdftitle={Global regular solutions for a one-dimensional transverse-field
  non-resistive MHD system with far-field vacuum},
  pdfauthor={Anonymous Author(s)},
  pdfsubject={Global regular solutions for one-dimensional degenerate
  non-resistive magnetohydrodynamics with far-field vacuum},
  pdfkeywords={compressible magnetohydrodynamics; density-dependent viscosity;
  far-field vacuum; non-resistive flow; effective velocity; large data},
  bookmarksnumbered=true
}

\allowdisplaybreaks[1]
\numberwithin{equation}{section}
\biboptions{sort&compress}
\newtheoremstyle{plainjournal}
  {6pt}{6pt}{\itshape}{}{\bfseries}{.}{0.5em}{}
\newtheoremstyle{definitionjournal}
  {6pt}{6pt}{\normalfont}{}{\bfseries}{.}{0.5em}{}

\theoremstyle{plainjournal}
\newtheorem{theorem}{Theorem}[section]
\newtheorem{proposition}[theorem]{Proposition}
\newtheorem{lemma}[theorem]{Lemma}
\newtheorem{corollary}[theorem]{Corollary}
\theoremstyle{definitionjournal}
\newtheorem{definition}[theorem]{Definition}
\newtheorem{remark}[theorem]{Remark}

\newcommand{\R}{\mathbb R}
\newcommand{\dd}{\mathop{}\!\mathrm d}
\newcommand{\Dt}{D_t}
\newcommand{\cD}{\mathcal D}
\newcommand{\norm}[2][]{\lVert #2\rVert_{#1}}
\newcommand{\ip}[2]{\langle #1,#2\rangle}

\begin{document}
\begin{frontmatter}

\title{Global Regular Solutions for One-Dimensional 
Degenerate Compressible Non-Resistive MHD Equations with Large Data and Far-Field Vacuum}

\author[label4]{Zhaoyang Shang\corref{cor1}}
\address[label4]{School of Finance, Shanghai Lixin University of Accounting and
Finance, Shanghai 201209, P.R. China}
\cortext[cor1]{Corresponding author.}
\ead{shangzhaoyang@sjtu.edu.cn}
\date{}

\begin{abstract}
We consider the Cauchy problem on the real line for a one-dimensional
transverse-field isentropic non-resistive magnetohydrodynamic system with
viscosity $\mu(\rho)=\alpha\rho^\delta$, $0<\delta\leq1$, and vacuum at
spatial infinity.  The viscosity therefore vanishes precisely where the
density vanishes.  We establish two global large-data theorems, one for the
sublinear range $0<\delta<1$ and one for the linear case $\delta=1$, in a
nontrivial invariant class of magnetic
fields.  The class is described by the frozen-in flux $h=b/\rho$ and its mass
derivative $k=\rho^{-1}h_x$.  The key identity
$[\partial_t+u\partial_x,\rho^{-1}\partial_x]=0$ transports both $h$ and $k$,
removes all negative powers of the density from the Lorentz force, and
propagates the magnetic regularity.  A modified effective-velocity entropy
contains an additional coercive magnetic term, while its only error is
controlled by $k$.  Magnetic pressure also strengthens the density upper
bound.  These estimates close the large-data continuation argument in the
ranges $0<\delta<1$, $\gamma\geq\delta+\tfrac12$, and $\delta=1$,
$\gamma>\tfrac32$. As a result, we extend the work of Cao, Li, and Zhu (SIAM J. Math. Anal., 2022) from the Navier-Stokes equations to the MHD equations. %The constant ratio $b_0=c\rho_0$ is recovered as a strict special case.
\end{abstract}


\end{frontmatter}

\section{Introduction}\label{sec:introduction}

We study a longitudinal flow carrying a transverse magnetic field,  the equations can be written as:
\begin{equation}\label{eq:system}
\left\{
\begin{aligned}
 &\rho_t+(\rho u)_x=0,\\
 &(\rho u)_t+\bigl(\rho u^2+A\rho^\gamma+\tfrac12|b|^2\bigr)_x
       =(\alpha\rho^\delta u_x)_x,\\
 &b_t+(ub)_x=0,
\end{aligned}\right.
\end{equation}
where $(t,x)\in(0,\infty)\times\R$, $\rho>0$ is the density, $u\in\R$ is the longitudinal velocity,
$b\in\R^m$ ($m=1$ or $2$) is the transverse magnetic field, and
$A,\alpha>0$, $\gamma>1$, $0<\delta\leq1$.  We impose
\begin{equation}\label{eq:initialdata}
 (\rho,u,b)|_{t=0}=(\rho_0,u_0,b_0),
\end{equation}
\begin{equation}\label{eq:farfield}
  (\rho,u,b)(t,x)\longrightarrow(0,0,0)
  \quad\text{as }|x|\to\infty.
\end{equation}
%and prescribe $(\rho,u,b)|_{t=0}=(\rho_0,u_0,b_0)$.

The system \eqref{eq:system} is a physically consistent invariant reduction
of planar ideal MHD: the longitudinal magnetic field is zero and the
transverse velocity is initially zero.  It is not the full planar
one-dimensional system with a nonzero longitudinal field.  The latter has an
Alfv\'en coupling with speed proportional to $\rho^{-1/2}$ and is not
covered by this paper.

Viscosities that vanish at vacuum make the momentum equation genuinely
degenerate.  The Bresch--Desjardins entropy structure
\cite{BreschDesjardins2003,BreschDesjardins2007} and its one-dimensional
effective-velocity form are central tools for such problems.  For
$\mu(\rho)=\alpha\rho^\delta$, Cao, Li and Zhu \cite{CaoLiZhu2022} proved
global regular solvability of the one-dimensional isentropic compressible
Navier--Stokes Cauchy problem for arbitrarily large data and far-field vacuum
when $0<\delta\leq1$.  Their work is the hydrodynamic point of departure for
the present paper.

Existing one-dimensional MHD results with vacuum concern different regimes.
Yu \cite{Yu2013} treated a zero-resistivity initial-boundary-value problem on
a bounded interval under an additional density--magnetic-field constraint.
Ye and Li \cite{YeLi2019} proved a Cauchy result with magnetic diffusion and
coercive viscosity.  Li, Wang and Ye \cite{LiWangYe2020} considered the
non-resistive Cauchy problem with far-field vacuum, but with a strictly
positive constant viscosity.  Guan and Wei \cite{GuanWei2025} recently
treated large data and initial vacuum on a bounded interval for
$\mu(\rho)=1+\rho^a$, which remains uniformly positive at vacuum.  Recent
multidimensional results with genuinely degenerate viscosity and far-field
vacuum are local in time; see Liu, Luo and Zhong \cite{LiuLuoZhong2025} and
Li, L\"u and Yuan \cite{LiLuYuan2026}.  To the best of our knowledge, none
of the cited results gives a global Cauchy theory combining far-field
vacuum, zero resistivity, and a viscosity that vanishes at vacuum for the
transverse-field invariant class considered here.

There are two immediate obstructions to transferring the hydrodynamic proof.
If
\begin{equation}\label{eq:vintro}
 v=u+\alpha\rho^{\delta-2}\rho_x,
\end{equation}
then
\begin{equation}\label{eq:vrawintro}
 \rho(v_t+uv_x)+(A\rho^\gamma+\tfrac12|b|^2)_x=0.
\end{equation}
The Lorentz acceleration $b\cdot b_x/\rho$ is apparently singular near
far-field vacuum.  Moreover, the effective-velocity entropy produces
\begin{equation}\label{eq:bdobstruction}
 \alpha\int_\R \rho^{\delta-2}\rho_x\,b\cdot b_x\,\dd x,
\end{equation}
which has no sign and is not controlled by the basic MHD energy.

Our resolution uses the frozen magnetic flux per unit mass and the mass
derivative
\begin{equation}\label{eq:hkintro}
 h=\frac b\rho,\qquad k=\cD h,\qquad \cD=\rho^{-1}\partial_x.
\end{equation}
Writing $\Dt=\partial_t+u\partial_x$, continuity and induction give
\begin{equation}\label{eq:commintro}
 \Dt h=0,\qquad [\Dt,\cD]=0,\qquad \Dt k=0.
\end{equation}
Consequently
\begin{equation}\label{eq:forceintro}
 \frac{b\cdot b_x}{\rho}
 =|h|^2\rho_x+\rho^2h\cdot k,
\end{equation}
which contains no negative power of $\rho$.  The density upper estimate is
also improved.  With
$\xi(t,x)=\int_{-\infty}^x\rho u(t,y)\,\dd y$, one has along particle paths
\begin{equation}\label{eq:densityintro}
 \Dt\left(\xi+\frac\alpha\delta\rho^\delta\right)
 =-A\rho^\gamma-\frac12|b|^2\leq0.
\end{equation}
Finally, the modified effective-velocity entropy has only one new error,
which is controlled by the transported $k$ and the density upper bound.

The magnetic data used below are genuinely nonconstant.  We assume that
$k_0$ belongs to $H^2$ and is compactly supported.  Compact support is used
only in the positive-density approximation at spatial infinity; it says that
$h_0$ is constant on each far tail, not that it is constant throughout the
fluid.  Removing this technical approximation hypothesis is an interesting
question, but it is not needed for the global mechanism.

For readers who are new to degenerate fluid equations, we summarize the
logic before beginning the estimates.  The proof has five layers.
\begin{enumerate}[label=\textup{(\arabic*)},leftmargin=2.2em]
\item The basic energy controls the total mass, kinetic energy, internal
energy, magnetic energy, and the weighted dissipation
$\rho^{\delta/2}u_x$.  It does not give a positive lower bound for the
density, and no such lower bound is used.
\item The integrated momentum equation gives a uniform upper bound for
$\rho$.  This upper bound is especially useful when $0<\delta<1$: since
$\delta-1<0$, it also gives a positive lower bound for the coefficient
$\rho^{\delta-1}$ appearing in the reformulated velocity equation.
\item The effective velocity $v$ cancels the differentiated viscous term.
The variables $h$ and $k$ turn the apparently singular Lorentz acceleration
into lower-order products of nonnegative powers of $\rho$.  Weighted velocity
moments then yield an integrable pointwise bound for $\rho^\iota u$, and the
ODE for $v$ gives $v\in L^\infty$.
\item Once $v$ is bounded, the identity
$\rho_x=\alpha^{-1}\rho^{2-\delta}(v-u)$ controls the first density
derivative.  The estimates are then obtained in the order
$u_t,u_{xx}$, followed by $\phi_{xx},\psi_x,u_{xxx}$, and finally the
time-weighted derivatives.  This order prevents the magnetic terms from
causing a derivative loss.
\item A positive-density approximation supplies a local solution with
bounds independent of the artificial density floor.  Compactness removes
the floor, and the finite-time estimates rule out breakdown at any finite
time.
\end{enumerate}

The paper is organized as follows.  We state the two main theorems in
\cref{sec:statement}.  The frozen mass-coordinate geometry is developed in
\cref{sec:geometry}.  Energy, density, and effective-velocity estimates are
proved in \cref{sec:basic,sec:effective}.  The high-order closure appears in
\cref{sec:higher}; local solvability and continuation are given in
\cref{sec:local}.  The two continuation proofs are completed in
\cref{sec:global}.  The positive-background construction and the full
difference argument are written out in \ref{app:local-construction}.

\section{Main results}\label{sec:statement}

We write $H^s=H^s(\R)$, $L^p=L^p(\R)$, and
$D^1=\{f\in L^1_{\mathrm{loc}}(\R):f_x\in L^2\}$, endowed with the
seminorm $|f|_{D^1}=\norm[L^2]{f_x}$.  Constants denoted by $C$
may change from line to line; $C_T$ may depend on $T$, the physical
parameters, and the stated initial norms, but not on a positive lower bound
for $\rho$.  Define
\begin{equation}\label{eq:psidef}
 \psi=\delta\rho^{\delta-2}\rho_x
 =
 \begin{cases}
 \displaystyle\frac{\delta}{\delta-1}(\rho^{\delta-1})_x,
     &0<\delta<1,\\[1ex]
 (\log\rho)_x,&\delta=1,
 \end{cases}
 \qquad
 \phi=\frac{A\gamma}{\gamma-1}\rho^{\gamma-1}.
\end{equation}

\begin{definition}[Regular MHD solution]\label{def:regular}
Let $T>0$.  A triple $(\rho,u,b)$ is a regular solution of
\eqref{eq:system}--\eqref{eq:farfield} on $[0,T]$ if the equations hold in
distributions, $\rho(t,x)>0$ at every finite $x$,
$\inf_{[0,T]\times\R}\rho=0$, and the following properties hold.

\smallskip
\noindent\emph{Sublinear regime $0<\delta<1$.}
\begin{align}
 &0<\rho^{\gamma-1}\in C([0,T];H^2),\qquad
   (\rho^{\delta-1})_x\in C([0,T];H^1),\label{eq:regdensity}\\
 &u\in C([0,T];H^2)\cap L^2(0,T;H^3),\nonumber\\
 &u_t\in C([0,T];L^2)\cap L^2(0,T;H^1),\nonumber\\
 &\rho^{(\delta-1)/2}u_x\in C([0,T];L^2),\qquad
   \rho^{(\delta-1)/2}u_{tx}\in L^2(0,T;L^2),\nonumber\\
 &\rho^{\delta-1}u_{xx}\in C([0,T];L^2)
       \cap L^2(0,T;H^1).\label{eq:regvelocity}
\end{align}
\smallskip
\noindent\emph{Linear regime $\delta=1$.}
In this case \eqref{eq:regdensity} is replaced by
\begin{equation}\label{eq:regdeltaone}
 0<\rho^{\gamma-1}\in C([0,T];H^2),\qquad
 (\log\rho)_x\in C([0,T];H^1),
\end{equation}
and, written without abbreviating the time interval,
\begin{equation}\label{eq:regvelocity-linear}
 u\in C([0,T];H^2)\cap L^2(0,T;H^3),\qquad
 u_t\in C([0,T];L^2)\cap L^2(0,T;H^1).
\end{equation}

\smallskip
\noindent\emph{Magnetic regularity common to both regimes.}
With $h=b/\rho$ and $k=\rho^{-1}h_x$,
\begin{align}
 h&\in L^\infty((0,T)\times\R),\qquad
 k\in C([0,T];H^2),\label{eq:regmagnetic}\\
 b&\in C([0,T];H^2\cap L^1),\qquad
 b_t\in C([0,T];H^1).\nonumber
\end{align}
\end{definition}

Thus, for example, $u\in C([0,T];H^2)$ means that the map
$t\mapsto u(t,\cdot)$ is continuous when $H^2(\R)$ is equipped with its
usual norm.  The condition $\inf\rho=0$ records far-field vacuum; it is not
in conflict with pointwise positivity, because the infimum is approached
only as $|x|\to\infty$.  The weighted quantities in
\eqref{eq:regvelocity} are precisely those that remain finite when the
viscosity degenerates at that vacuum.

We now state the two parameter regimes separately.  This is useful because
the quantities that must be controlled at vacuum are different.  When
$0<\delta<1$, the factors $\rho^{\delta-1}$ are singular as
$\rho\to0$, so weighted compatibility conditions are needed.  When
$\delta=1$, these weights are equal to one and the corresponding statement
is genuinely unweighted.

\begin{theorem}[Global regular solutions for $0<\delta<1$]
\label{thm:global-sublinear}
Let $A,\alpha>0$ and assume
\begin{equation}\label{eq:parameters-sublinear}
 0<\delta<1,\qquad \gamma>1,
 \qquad \gamma\geq\delta+\frac12.
\end{equation}
Suppose that the initial density and velocity satisfy
\begin{equation}\label{eq:datafluid-sublinear}
 0<\rho_0\in L^1(\R),\qquad
 \rho_0^{\gamma-1}\in H^2(\R),\qquad u_0\in H^2(\R),
\end{equation}
and that
\begin{equation}\label{eq:datadelta}
 (\rho_0^{\delta-1})_x\in H^1(\R),\qquad
 u_{0x}=\rho_0^{(1-\delta)/2}g_1,\qquad
 \alpha u_{0xx}=\rho_0^{1-\delta}g_2
\end{equation}
for some $g_1,g_2\in L^2(\R)$.  Write $b_0=\rho_0h_0$ and assume
\begin{equation}\label{eq:datamagnetic-sublinear}
 \begin{split}
 h_0&\in L^\infty(\R;\R^m)
       \cap W^{1,1}_{\mathrm{loc}}(\R;\R^m),\\
 k_0&:=\rho_0^{-1}(h_0)_x\in H^2(\R;\R^m),
 \qquad \operatorname{supp}k_0\Subset\R.
 \end{split}
\end{equation}
The identity $(h_0)_x=\rho_0k_0$ in
\eqref{eq:datamagnetic-sublinear} is understood distributionally.  Because
$h_0$ is locally absolutely continuous and $\rho_0$ is positive at every
finite point, the quotient defining $k_0$ is well defined almost
everywhere.

Then the Cauchy problem \eqref{eq:system}--\eqref{eq:farfield}, with
$(\rho,u,b)|_{t=0}=(\rho_0,u_0,b_0)$, has a unique global classical
solution on $(0,\infty)\times\R$.  For each finite $T>0$ this solution has
the sublinear regularity in \cref{def:regular} and satisfies
\begin{align}
 &\rho\in C([0,T];L^1),
 \qquad (\rho^{\gamma-1})_t\in C([0,T];H^1),
 \label{eq:extra-sublinear}\\
 &h(t,X(t;x))=h_0(x),\qquad
 k(t,X(t;x))=k_0(x),
 \label{eq:transport-sublinear}
\end{align}
where the particle path is determined by
\begin{equation}\label{eq:flow-sublinear}
 X_t(t;x)=u(t,X(t;x)),\qquad X(0;x)=x.
\end{equation}
Moreover,
\begin{align}
 &t^{1/2}\rho^{(\delta-1)/2}u_{tx}
       \in L^\infty(0,T;L^2),
 \label{eq:timeweighted-sublinear}\\
 &t^{1/2}u_{txx}\in L^2(0,T;L^2),\qquad
 t^{1/2}u_{tt}\in L^2(0,T;L^2).\nonumber
\end{align}
No smallness assumption is imposed on any of the norms above.
\end{theorem}

\begin{theorem}[Global regular solutions for $\delta=1$]
\label{thm:global-linear}
Let $A,\alpha>0$ and assume
\begin{equation}\label{eq:parameters-linear}
 \delta=1,\qquad \gamma>\frac32.
\end{equation}
Suppose that the initial density and velocity satisfy
\begin{equation}\label{eq:datafluid-linear}
 0<\rho_0\in L^1(\R),\qquad
 \rho_0^{\gamma-1}\in H^2(\R),\qquad u_0\in H^2(\R),
 \qquad (\log\rho_0)_x\in H^1(\R).
\end{equation}
Write $b_0=\rho_0h_0$ and assume
\begin{equation}\label{eq:datamagnetic-linear}
 \begin{split}
 h_0&\in L^\infty(\R;\R^m)
       \cap W^{1,1}_{\mathrm{loc}}(\R;\R^m),\\
 k_0&:=\rho_0^{-1}(h_0)_x\in H^2(\R;\R^m),
 \qquad \operatorname{supp}k_0\Subset\R.
 \end{split}
\end{equation}
Again, $(h_0)_x=\rho_0k_0$ is understood in the distributional sense.

Then the Cauchy problem \eqref{eq:system}--\eqref{eq:farfield}, with
$(\rho,u,b)|_{t=0}=(\rho_0,u_0,b_0)$, has a unique global classical
solution on $(0,\infty)\times\R$.  For each finite $T>0$ this solution has
the $\delta=1$ regularity in \cref{def:regular} and satisfies
\begin{align}
 &\rho\in C([0,T];L^1),
 \qquad (\rho^{\gamma-1})_t\in C([0,T];H^1),
 \label{eq:extra-linear}\\
 &h(t,X(t;x))=h_0(x),\qquad
 k(t,X(t;x))=k_0(x),
 \label{eq:transport-linear}
\end{align}
where
\begin{equation}\label{eq:flow-linear}
 X_t(t;x)=u(t,X(t;x)),\qquad X(0;x)=x.
\end{equation}
The time-weighted estimates are now completely unweighted:
\begin{align}
 &t^{1/2}u_{tx}\in L^\infty(0,T;L^2),
 \label{eq:timeweighted-linear}\\
 &t^{1/2}u_{txx}\in L^2(0,T;L^2),\qquad
 t^{1/2}u_{tt}\in L^2(0,T;L^2).\nonumber
\end{align}
No smallness assumption is imposed on the data.
\end{theorem}

\begin{remark}[Nonconstant admissible fields]\label{rem:nonempty}
Let $k_0\in C_c^\infty(\R;\R^m)$ and $c\in\R^m$, and define
\begin{equation}\label{eq:exampleh}
 h_0(x)=c+\int_{-\infty}^x\rho_0(y)k_0(y)\,\dd y,\qquad
 b_0=\rho_0h_0.
\end{equation}
Then the magnetic assumptions in both
\cref{thm:global-sublinear,thm:global-linear} hold.  Whenever
$k_0\not\equiv0$, the ratio $b_0/\rho_0$ is nonconstant, so neither main
theorem is merely a barotropic reduction.
\end{remark}

\begin{corollary}[Constant flux-to-mass ratio]\label{cor:constant}
Assume the fluid hypotheses of \cref{thm:global-sublinear} or
\cref{thm:global-linear}, and let $b_0=c\rho_0$ for a
constant $c\in\R^m$.  Then $b=c\rho$ for all time and the MHD system reduces
to the degenerate Navier--Stokes system with compound pressure
\begin{equation}\label{eq:compoundpressure}
 P_{\mathrm{eff}}(\rho)=A\rho^\gamma+\tfrac12|c|^2\rho^2.
\end{equation}
If $0<\delta<1$, the conclusion of \cref{thm:global-sublinear} holds; if
$\delta=1$, the conclusion of \cref{thm:global-linear} holds.
\end{corollary}

\section{Frozen flux and mass-coordinate derivatives}\label{sec:geometry}

The identities in this section are first justified for smooth positive
solutions and then pass to regular solutions by approximation.

\begin{lemma}[Commuting material and mass derivatives]\label{lem:commute}
Let $\Dt=\partial_t+u\partial_x$ and $\cD=\rho^{-1}\partial_x$.  If
$\rho_t+(\rho u)_x=0$, then
\begin{equation}\label{eq:commutator}
 [\Dt,\cD]f=0
\end{equation}
for every smooth scalar or vector function $f$.  If also
$b_t+(ub)_x=0$, then
\begin{equation}\label{eq:transporthk}
 \Dt h=0,\qquad \Dt k=0,\qquad h=b/\rho,\quad k=\cD h.
\end{equation}
Consequently,
\begin{equation}\label{eq:hkmax}
 \norm[L^\infty]{h(t)}=\norm[L^\infty]{h_0},\qquad
 \norm[L^\infty]{k(t)}=\norm[L^\infty]{k_0}.
\end{equation}
\end{lemma}

\begin{proof}
Expanding the divergence in the continuity equation gives
\begin{equation}\label{eq:materialrho}
 \Dt\rho=\rho_t+u\rho_x=-\rho u_x,
 \qquad \Dt(\rho^{-1})=\rho^{-1}u_x.
\end{equation}
The derivative $\partial_x$ does not commute with $\Dt$.  Indeed, the
ordinary product rule gives
\begin{equation}\label{eq:materialfx}
 \Dt f_x=f_{xt}+uf_{xx}=(\Dt f)_x-u_xf_x.
\end{equation}
Combining \eqref{eq:materialrho} and \eqref{eq:materialfx}, we obtain
\begin{align*}
 \Dt(\rho^{-1}f_x)
 &=\Dt(\rho^{-1})f_x+\rho^{-1}\Dt f_x\\
 &=\rho^{-1}u_xf_x
   +\rho^{-1}\bigl((\Dt f)_x-u_xf_x\bigr)
 =\rho^{-1}(\Dt f)_x=\cD(\Dt f).
\end{align*}
This is exactly the commutator identity \eqref{eq:commutator}.

The induction equation can be written
\begin{equation}\label{eq:materialb}
 \Dt b=-bu_x.
\end{equation}
Since $h=b/\rho$, another application of the quotient rule gives the
cancellation
\[
 \Dt h=\frac{\Dt b}{\rho}-\frac b{\rho^2}\Dt\rho
 =-\frac b\rho u_x+\frac b\rho u_x=0.
\]
Because $k=\cD h$, the already proved commutator yields
$\Dt k=\cD(\Dt h)=0$.

For completeness, let $X$ be the particle map
$X_t(t;a)=u(t,X(t;a))$, $X(0;a)=a$.  Differentiating it with respect to
$a$ gives
\begin{equation}\label{eq:flowjacobian}
 X_a(t;a)=\exp\left(\int_0^t
 u_x(s,X(s;a))\,\dd s\right)>0.
\end{equation}
The two transport equations therefore imply
\begin{equation}\label{eq:hkcharacteristics}
 h(t,X(t;a))=h_0(a),\qquad k(t,X(t;a))=k_0(a).
\end{equation}
Taking suprema proves \eqref{eq:hkmax}.  Notice that this argument preserves
the $L^\infty$ norms, but not the $H^2$ norm of $k$; deformation of the
flow is accounted for separately in \cref{lem:kH2}.
\end{proof}

\begin{lemma}[Propagation of the magnetic profile]\label{lem:kH2}
For every $T>0$,
\begin{equation}\label{eq:kH2bound}
 \sup_{0\leq t\leq T}\norm[H^2]{k(t)}
 \leq \norm[H^2]{k_0}
 \exp\left\{C\int_0^T
 \bigl(\norm[L^\infty]{u_x}+\norm[L^2]{u_{xx}}\bigr)\,\dd t\right\}.
\end{equation}
Moreover, $\operatorname{supp}k(t)$ is the image of
$\operatorname{supp}k_0$ under the flow.
\end{lemma}

\begin{proof}
Starting from $k_t+uk_x=0$, differentiate once and twice in $x$:
\begin{align}
 (k_x)_t+u k_{xx}+u_xk_x&=0,
 \label{eq:kxtransport}\\
 (k_{xx})_t+u k_{xxx}+2u_xk_{xx}+u_{xx}k_x&=0.
 \label{eq:kxxtransport}
\end{align}
Taking the $L^2$ inner products of these three equations with
$k,k_x,k_{xx}$, respectively, and using
$-\int u f_x\cdot f=\frac12\int u_x|f|^2$, gives
\begin{align}
 \frac12\frac{\dd}{\dd t}\norm[L^2]{k}^2
 &=\frac12\int_\R u_x|k|^2\,\dd x,\label{eq:k0energy}\\
 \frac12\frac{\dd}{\dd t}\norm[L^2]{k_x}^2
 &=-\frac12\int_\R u_x|k_x|^2\,\dd x,\label{eq:k1energy}\\
 \frac12\frac{\dd}{\dd t}\norm[L^2]{k_{xx}}^2
 &=-\frac32\int_\R u_x|k_{xx}|^2\,\dd x
   -\int_\R u_{xx}k_x\cdot k_{xx}\,\dd x.\label{eq:k2energy}
\end{align}
For the last term in \eqref{eq:k2energy}, H\"older's inequality and the
one-dimensional embedding $H^1(\R)\hookrightarrow L^\infty(\R)$ give
\[
 \left|\int u_{xx}k_x\cdot k_{xx}\right|
 \leq\norm[L^2]{u_{xx}}\norm[L^\infty]{k_x}
       \norm[L^2]{k_{xx}}
 \leq C\norm[L^2]{u_{xx}}\norm[H^2]{k}^2.
\]
The other terms are bounded by
$\norm[L^\infty]{u_x}\norm[H^2]{k}^2$.  Adding the three identities yields
\[
 \frac{\dd}{\dd t}\norm[H^2]{k}^2
 \leq C\bigl(\norm[L^\infty]{u_x}+\norm[L^2]{u_{xx}}\bigr)
       \norm[H^2]{k}^2.
\]
Integrating this scalar differential inequality and applying Gronwall's
lemma proves \eqref{eq:kH2bound}.  The support statement follows from the
transport formulas in \cref{lem:commute}: a characteristic that starts
outside $\operatorname{supp}k_0$ carries the value zero for all later times.
\end{proof}

\begin{lemma}[Regular Lorentz-force representation]\label{lem:force}
Let
\[
 \mathcal M:=\frac{b\cdot b_x}{\rho}.
\]
Then
\begin{align}
 h_x&=\rho k,\label{eq:hx}\\
 \mathcal M&=|h|^2\rho_x+\rho^2h\cdot k,\label{eq:Mrepresentation}\\
 \mathcal M_x&=|h|^2\rho_{xx}
 +4\rho(h\cdot k)\rho_x+\rho^3|k|^2+\rho^2h\cdot k_x.
 \label{eq:Mxrepresentation}
\end{align}
If $0<\rho\leq R$, $\rho\in H^2\cap L^1$,
$\norm[L^\infty]{h}\leq R$, and $\norm[H^2]{k}\leq R$, then
\begin{equation}\label{eq:sourceH1}
 \norm[H^2]{b}+\norm[H^1]{\mathcal M}
 \leq C_R\bigl(1+\norm[H^2]{\rho}\bigr).
\end{equation}
\end{lemma}

\begin{proof}
Since $b=\rho h$ and $h_x=\rho k$,
\[
 b_x=h\rho_x+\rho^2k,\qquad
 b_{xx}=h\rho_{xx}+3\rho k\rho_x+\rho^2k_x.
\]
Taking the scalar product of $b$ with $b_x$ gives
\eqref{eq:Mrepresentation}; differentiation gives
\eqref{eq:Mxrepresentation}.  The one-dimensional Sobolev embedding,
$\rho\in L^1\cap L^\infty$, and these formulas prove
\eqref{eq:sourceH1}.
\end{proof}

\begin{lemma}[Low-order Lorentz-force bound]\label{lem:Mlow}
Assume $0<\delta\leq1$, $\rho\in L^1\cap L^\infty$,
$\sqrt\rho\,u\in L^2$, $v\in L^\infty$, $h,k\in L^\infty$.  Then
\begin{equation}\label{eq:Mlow}
 \norm[L^2]{\mathcal M}
 \leq C\left(1+\norm[L^\infty]{v}
 +\norm[L^2]{\sqrt\rho\,u}\right).
\end{equation}
\end{lemma}

\begin{proof}
Since $\rho_x=\alpha^{-1}\rho^{2-\delta}(v-u)$,
\[
 \mathcal M=\alpha^{-1}|h|^2\rho^{2-\delta}(v-u)
             +\rho^2h\cdot k.
\]
The terms without $u$ are square integrable by
$\rho\in L^1\cap L^\infty$.  For the remaining term,
\[
 \int_\R\rho^{4-2\delta}|u|^2\,\dd x
 \leq\norm[L^\infty]{\rho}^{3-2\delta}
       \int_\R\rho|u|^2\,\dd x,
\]
because $3-2\delta\geq1$.  This proves the result.
\end{proof}

\section{Energy and the density upper bound}\label{sec:basic}

\begin{lemma}[Mass and total energy]\label{lem:energy}
Every regular solution satisfies
\begin{equation}\label{eq:mass}
 \int_\R\rho(t,x)\,\dd x=M_0:=\int_\R\rho_0\,\dd x
\end{equation}
and
\begin{equation}\label{eq:energy}
 E(t)+\alpha\int_0^t\!\int_\R\rho^\delta u_x^2\,\dd x\dd s=E(0),
\end{equation}
where
\begin{equation}\label{eq:Edef}
 E(t)=\int_\R\left(\frac12\rho u^2
 +\frac{A}{\gamma-1}\rho^\gamma+\frac12|b|^2\right)\dd x.
\end{equation}
\end{lemma}

\begin{proof}
We first perform the calculation for a smooth solution that decays rapidly;
the cutoff argument at the end makes it valid for a regular solution.  Put
\[
 p=A\rho^\gamma,\qquad e=\frac{A}{\gamma-1}\rho^\gamma,
 \qquad B=\frac12|b|^2.
\]
The continuity equation and the chain rule give the internal-energy identity
\begin{equation}\label{eq:internal-local}
 e_t+(ue)_x+pu_x=0.
\end{equation}
Indeed, after expanding $(ue)_x$, its left-hand side is
$\frac{A\gamma}{\gamma-1}\rho^{\gamma-1}
(\rho_t+u\rho_x)+\frac{A\gamma}{\gamma-1}\rho^\gamma u_x$,
which vanishes because $\Dt\rho=-\rho u_x$.

Writing the induction equation as $b_t+ub_x+u_xb=0$ and taking its scalar
product with $b$ gives
\begin{equation}\label{eq:magnetic-local}
 B_t+(uB)_x+Bu_x=0.
\end{equation}
Next multiply the momentum equation by $u$.  The continuity equation turns
the first two terms into a kinetic-energy derivative:
\begin{equation}\label{eq:kinetic-local}
 \left(\frac12\rho u^2\right)_t
 +\left(\frac12\rho u^3\right)_x
 +up_x+uB_x=u(\alpha\rho^\delta u_x)_x.
\end{equation}
Use
\[
 up_x=(up)_x-pu_x,\quad uB_x=(uB)_x-Bu_x,\quad
 u(\alpha\rho^\delta u_x)_x
 =(\alpha\rho^\delta uu_x)_x-\alpha\rho^\delta u_x^2.
\]
Adding \eqref{eq:internal-local}--\eqref{eq:kinetic-local} cancels both
$pu_x$ and $Bu_x$ and yields the local identity
\begin{align}
 &\partial_t\left(\frac12\rho u^2+e+B\right)
 +\alpha\rho^\delta u_x^2+\partial_x\left[
 \frac12\rho u^3+u(e+p)+2uB-\alpha\rho^\delta uu_x
 \right]=0.\label{eq:total-local}
\end{align}
This line displays explicitly why the magnetic work from momentum is
cancelled by the magnetic energy from induction.

Let $\chi_R(x)=\chi(x/R)$, where $0\leq\chi\leq1$, $\chi=1$ on
$[-1,1]$, and $\chi=0$ outside $[-2,2]$.  Integrating continuity and
\eqref{eq:total-local} against $\chi_R$ leaves flux errors supported in
$R\leq|x|\leq2R$ and multiplied by $\norm[L^\infty]{\chi_R'}\leq C/R$.
The regularity in \cref{def:regular}, followed by approximation if needed,
makes those errors tend to zero as $R\to\infty$.  We obtain
\eqref{eq:mass} and \eqref{eq:energy}.
\end{proof}

\begin{lemma}[Uniform upper bound for the density]\label{lem:densityupper}
There is a constant $C_0$, depending only on $E(0)$, $M_0$, $\alpha$,
$\delta$, and $\norm[L^\infty]{\rho_0}$, such that
\begin{equation}\label{eq:rhoinfty}
 0<\rho(t,x)\leq C_0,\qquad (t,x)\in[0,\infty)\times\R.
\end{equation}
\end{lemma}

\begin{proof}
Set $\xi(t,x)=\int_{-\infty}^x\rho u(t,y)\,\dd y$.  The momentum
fluxes vanish at the left far field.  Integration from $-\infty$ to $x$
therefore gives
\begin{equation}\label{eq:xit}
 \xi_t=-\rho u^2-A\rho^\gamma-\frac12|b|^2
       +\alpha\rho^\delta u_x.
\end{equation}
Since $\xi_x=\rho u$, adding $u\xi_x=\rho u^2$ to
\eqref{eq:xit} cancels the convective term:
\begin{equation}\label{eq:Dtxi}
 \Dt\xi=\alpha\rho^\delta u_x-A\rho^\gamma-\frac12|b|^2.
\end{equation}
On the other hand, \eqref{eq:materialrho} and the chain rule yield
\begin{equation}\label{eq:Dtrhodelta}
 \Dt\left(\frac\alpha\delta\rho^\delta\right)
 =\alpha\rho^{\delta-1}\Dt\rho=-\alpha\rho^\delta u_x.
\end{equation}
Adding \eqref{eq:Dtxi} and \eqref{eq:Dtrhodelta} proves
\eqref{eq:densityintro}.  If $X(t;x)$ is the particle that starts at $x$,
integration of this nonpositive derivative from $0$ to $t$ gives
\[
 \frac\alpha\delta\rho^\delta(t,X(t;x))
 \leq \frac\alpha\delta\rho_0^\delta(x)
      +|\xi_0(x)|+|\xi(t,X(t;x))|.
\]
For every $x$, Cauchy--Schwarz, mass conservation, and energy give
\[
 |\xi(t,x)|\leq\int_\R\rho|u|\,\dd y
 \leq M_0^{1/2}\norm[L^2]{\sqrt\rho\,u(t)}
 \leq (2M_0E(0))^{1/2}.
\]
The same estimate holds at $t=0$.  Taking the $\delta$th root proves
\eqref{eq:rhoinfty}, with a constant independent of time.  No estimate of
$|b|^2$ was needed in this argument: magnetic pressure appears with a minus
sign in \eqref{eq:densityintro} and can only improve the upper bound.
\end{proof}

\begin{lemma}[Weighted velocity moments]\label{lem:moments}
For every $p\geq0$ and $T>0$,
\begin{equation}\label{eq:momentbound}
 \sup_{0\leq t\leq T}\int_\R\rho|u|^{p+2}\,\dd x
 +\int_0^T\!\int_\R\rho^\delta|u|^p u_x^2\,\dd x\dd t
 \leq C_{T,p}.
\end{equation}
\end{lemma}

\begin{proof}
We first take $p$ so that the multiplier is smooth; a standard regularization
of $|u|^p u$ removes this harmless restriction.  Multiply the conservative
momentum equation by $|u|^pu$ and integrate.  The continuity equation
combines the time and convection terms, while
$(|u|^pu)_x=(p+1)|u|^pu_x$.  Thus
\begin{align}
 &\frac1{p+2}\frac{\dd}{\dd t}
    \int_\R\rho|u|^{p+2}\,\dd x
 +\alpha(p+1)\int_\R\rho^\delta|u|^pu_x^2\,\dd x
=(p+1)\int_\R
 \left(A\rho^\gamma+\frac12|b|^2\right)|u|^pu_x\,\dd x.
 \label{eq:momentidentity}
\end{align}
Young's inequality with the dissipative factor
$\rho^{\delta/2}|u|^{p/2}u_x$ gives
\begin{align}
 \left|\int_\R A\rho^\gamma|u|^pu_x\,\dd x\right|
 &\leq\varepsilon\int_\R\rho^\delta|u|^pu_x^2\,\dd x
 +C_{\varepsilon,p}\int_\R\rho^{2\gamma-\delta}|u|^p\,\dd x,
 \label{eq:momentpressure}\\
 \left|\int_\R\frac{|b|^2}{2}|u|^pu_x\,\dd x\right|
 &\leq\varepsilon\int_\R\rho^\delta|u|^pu_x^2\,\dd x
 +C_{\varepsilon,p}\int_\R\rho^{4-\delta}|u|^p\,\dd x.
 \label{eq:momentmag}
\end{align}
Here we used $|b|=\rho|h|\leq\norm[L^\infty]{h_0}\rho$.

We spell out the last integrability step.  If $q>1$, split the line into
$\{|u|\leq1\}$ and $\{|u|>1\}$.  The density upper bound implies
\[
 \rho^q|u|^p
 \leq C\rho\quad\hbox{on }\{|u|\leq1\},
 \qquad
 \rho^q|u|^p
 \leq C\rho|u|^{p+2}\quad\hbox{on }\{|u|>1\}.
\]
Consequently,
\begin{equation}\label{eq:momentlower}
 \int_\R\rho^q|u|^p\,\dd x
 \leq C\left(M_0+\int_\R\rho|u|^{p+2}\,\dd x\right).
\end{equation}
Both exponents $2\gamma-\delta$ and $4-\delta$ are larger than one.
Choose $\varepsilon$ small enough to absorb the two dissipative terms in
\eqref{eq:momentidentity}, use \eqref{eq:momentlower}, and apply Gronwall's
lemma.  The initial moment is finite because $u_0\in H^2(\R)$ is bounded
and $\rho_0\in L^1$.  This proves \eqref{eq:momentbound} for every $p\geq0$.
\end{proof}

\section{Effective velocity and modified entropy}\label{sec:effective}

\begin{lemma}[Magnetic effective-velocity entropy]\label{lem:BD}
Let $v$ be given by \eqref{eq:vintro}.  Then
\begin{align}
 &\frac{\dd}{\dd t}\int_\R
 \left(\frac12\rho|v|^2+\frac{A}{\gamma-1}\rho^\gamma
       +\frac12\rho^2|h|^2\right)\dd x
 +\alpha A\gamma\int_\R\rho^{\gamma+\delta-3}\rho_x^2\,\dd x+\alpha\int_\R |h|^2\rho^{\delta-1}\rho_x^2\,\dd x
 \nonumber\\
 &=-\alpha\int_\R\rho^{\delta+1}(h\cdot k)\rho_x\,\dd x.
 \label{eq:BDidentity}
\end{align}
Consequently, for every $T>0$,
\begin{align}
 &\sup_{0\leq t\leq T}\int_\R
 \left(\rho|v|^2+\rho^\gamma+|b|^2\right)\dd x
 +\int_0^T\!\int_\R\rho^{\gamma+\delta-3}\rho_x^2\,\dd x\dd t
 +\int_0^T\!\int_\R |h|^2\rho^{\delta-1}\rho_x^2\,\dd x\dd t
 \leq C_T.
 \label{eq:BDbound}
\end{align}
\end{lemma}

\begin{proof}
Put $w=\alpha\rho^{\delta-2}\rho_x$, so that $v=u+w$.  We first verify the
viscous cancellation instead of treating it as a formal identity.  From
\eqref{eq:materialrho} and \eqref{eq:materialfx},
\begin{equation}\label{eq:Dtrhox}
 \Dt\rho_x=(\Dt\rho)_x-u_x\rho_x
 =-\rho u_{xx}-2u_x\rho_x.
\end{equation}
Therefore
\begin{align*}
 \rho\Dt w
 &=\alpha\rho\Dt(\rho^{\delta-2}\rho_x)=-\alpha\rho^\delta u_{xx}
   -\alpha\delta\rho^{\delta-1}\rho_xu_x
 =-(\alpha\rho^\delta u_x)_x.
\end{align*}
In other words, continuity implies the exact identity
\begin{equation}\label{eq:visccancel}
 \rho\Dt(\alpha\rho^{\delta-2}\rho_x)
 =-(\alpha\rho^\delta u_x)_x.
\end{equation}
The nonconservative momentum equation is
$\rho\Dt u+(A\rho^\gamma+|b|^2/2)_x
=(\alpha\rho^\delta u_x)_x$.  Adding \eqref{eq:visccancel} cancels its
right-hand side and gives \eqref{eq:vrawintro}.

Multiply that equation by $v=u+w$ and integrate.  Continuity gives
\[
 \int_\R\rho v\cdot\Dt v\,\dd x
 =\frac{\dd}{\dd t}\int_\R\frac12\rho|v|^2\,\dd x.
\]
The $u$ part of the pressure work is the positive time derivative of the
internal and magnetic energies, by \eqref{eq:internal-local} and
\eqref{eq:magnetic-local}.  For the $w$ part, integration by parts gives
\begin{align*}
 \int_\R\left(A\rho^\gamma+\frac12|b|^2\right)w_x\,\dd x
 ={}&-\alpha A\gamma\int_\R
 \rho^{\gamma+\delta-3}\rho_x^2\,\dd x\\
 &-\alpha\int_\R\rho^{\delta-2}\rho_x\,b\cdot b_x\,\dd x.
\end{align*}
Combining these identities yields
\begin{align*}
 \frac{\dd}{\dd t}\int_\R
 \left(\frac12\rho v^2+\frac{A}{\gamma-1}\rho^\gamma
 +\frac12|b|^2\right)\dd x
 ={}&-\alpha A\gamma\int_\R
 \rho^{\gamma+\delta-3}\rho_x^2\,\dd x
 -\alpha\int_\R\rho^{\delta-2}\rho_x\,b\cdot b_x\,\dd x.
\end{align*}
To understand the magnetic term, use $b=\rho h$ and $h_x=\rho k$:
\begin{equation}\label{eq:bdmagneticexpand}
 b\cdot b_x=\rho|h|^2\rho_x+\rho^3h\cdot k.
\end{equation}
Thus
\[
 \rho^{\delta-2}\rho_x\,b\cdot b_x
 =|h|^2\rho^{\delta-1}\rho_x^2
  +\rho^{\delta+1}(h\cdot k)\rho_x,
\]
which converts the last integral into the final two terms in
\eqref{eq:BDidentity}.  The second one has no sign, but Young's inequality
gives
\begin{equation}\label{eq:BDyoung}
 \rho^{\delta+1}|h|\,|k|\,|\rho_x|
 \leq\frac12|h|^2\rho^{\delta-1}\rho_x^2
      +\frac12\rho^{\delta+3}|k|^2.
\end{equation}
The error on the right is uniformly bounded because
\begin{equation}\label{eq:bdremainder}
 \int_\R\rho^{\delta+3}|k|^2\,\dd x
 \leq\norm[L^\infty]{k_0}^2
       \norm[L^\infty]{\rho}^{\delta+2}M_0.
\end{equation}
Absorb one half of the magnetic gradient term, integrate from $0$ to $t$,
and use \eqref{eq:rhoinfty}.  The initial effective-velocity energy is
finite by the corresponding assumptions on
$(\rho_0^{\delta-1})_x$ or $(\log\rho_0)_x$.  This proves
\eqref{eq:BDbound}.  In particular, the cross term is controlled; it is not
being discarded or assigned a favorable sign.
\end{proof}

Dividing \eqref{eq:vrawintro} by $\rho$ and using
\eqref{eq:Mrepresentation} yields
\begin{equation}\label{eq:vequation}
 \Dt v+\left(\frac{A\gamma}{\alpha}\rho^{\gamma-\delta}
 +\frac{|h|^2}{\alpha}\rho^{2-\delta}\right)(v-u)
 =-\rho^2h\cdot k.
\end{equation}

\begin{lemma}[Time-integrable weighted pointwise velocity]
\label{lem:weightedpointwise}
For every finite $T>0$, the following estimates hold:
\begin{align}
 &\int_0^T\norm[L^\infty]{\rho^{1/2}u}\,\dd t\leq C_T,
 &&0<\delta<1,\label{eq:rhouint-sublinear}\\
 &\int_0^T\norm[L^\infty]{\rho^\iota u}\,\dd t\leq C_{T,\iota},
 &&\delta=1,\quad \frac12<\iota<1.
 \label{eq:rhouint-linear}
\end{align}
\end{lemma}

\begin{proof}
We prove the two cases separately because they use different spatial
exponents.

\smallskip
\noindent\emph{Case $0<\delta<1$.}
Let $F=\rho^{1/2}u$ and set $p_\delta=2/(2-\delta)$.  A useful factorization
of its derivative is
\begin{equation}\label{eq:Fx-sublinear}
 F_x=\frac12
 \bigl(\rho^{\delta-3/2}\rho_x\bigr)
 \bigl(\rho^{1-\delta}u\bigr)
 +\rho^{(1-\delta)/2}\bigl(\rho^{\delta/2}u_x\bigr).
\end{equation}
The effective-velocity and basic energies imply
\begin{equation}\label{eq:rhograd-weighted}
 \norm[L^2]{\rho^{\delta-3/2}\rho_x}
 =\alpha^{-1}\norm[L^2]{\sqrt\rho\,(v-u)}\leq C_T.
\end{equation}
In H\"older's inequality choose the exponents $2$ and
$2/(1-\delta)$.  Since
\[
 \int_\R|\rho^{1-\delta}u|^{2/(1-\delta)}\,\dd x
 =\int_\R\rho^2|u|^{2/(1-\delta)}\,\dd x,
\]
the first factor is controlled by \cref{lem:moments} and the density upper
bound.  The second product in \eqref{eq:Fx-sublinear} is controlled because
$\rho^{(1-\delta)/2}\in L^{2/(1-\delta)}$ by mass conservation and
$\rho^{\delta/2}u_x\in L^2(0,T;L^2)$ by energy.  Hence
\begin{equation}\label{eq:Fx-sublinear-bound}
 \int_0^T\norm[L^{p_\delta}]{F_x}^2\,\dd t\leq C_T.
\end{equation}
The one-dimensional Gagliardo--Nirenberg inequality reads
\begin{equation}\label{eq:GN-sublinear}
 \norm[L^\infty]{F}
 \leq C\norm[L^2]{F}^{1-\theta_\delta}
       \norm[L^{p_\delta}]{F_x}^{\theta_\delta},
 \qquad \theta_\delta=\frac1{1+\delta}.
\end{equation}
The $L^2$ norm of $F$ is bounded by the kinetic energy.  Integrating
\eqref{eq:GN-sublinear} in time and using
\eqref{eq:Fx-sublinear-bound} proves \eqref{eq:rhouint-sublinear}.

\smallskip
\noindent\emph{Case $\delta=1$.}
Fix $1/2<\iota<1$ and put $F=\rho^\iota u$.  This time
\begin{equation}\label{eq:Fx-linear}
 F_x=\iota(\rho^{-1/2}\rho_x)
       (\rho^{\iota-1/2}u)
 +\rho^{\iota-1/2}(\rho^{1/2}u_x).
\end{equation}
By the same effective-velocity estimate,
$\norm[L^2]{\rho^{-1/2}\rho_x}\leq C_T$.  Use H\"older with exponents
$2$ and $2/(2\iota-1)$.  The identities
\[
 \int_\R|\rho^{\iota-1/2}u|^{2/(2\iota-1)}\,\dd x
 =\int_\R\rho|u|^{2/(2\iota-1)}\,\dd x,
 \qquad
 \int_\R|\rho^{\iota-1/2}|^{2/(2\iota-1)}\,\dd x=M_0
\]
and \cref{lem:moments} show that
\begin{equation}\label{eq:Fx-linear-bound}
 \int_0^T\norm[L^{1/\iota}]{F_x}^2\,\dd t\leq C_{T,\iota}.
\end{equation}
Moreover, $\norm[L^2]{F}$ is controlled by kinetic energy and
\eqref{eq:rhoinfty}, since $2\iota-1>0$.  The corresponding
Gagliardo--Nirenberg inequality is
\[
 \norm[L^\infty]{F}
 \leq C\norm[L^2]{F}^{1-\theta_\iota}
       \norm[L^{1/\iota}]{F_x}^{\theta_\iota},
 \qquad \theta_\iota=\frac1{3-2\iota}.
\]
Time integration proves \eqref{eq:rhouint-linear}.
\end{proof}

\begin{lemma}[Bounded effective velocity]\label{lem:vinfty}
Under the parameter restrictions of \cref{thm:global-sublinear} or
\cref{thm:global-linear}, for every $T>0$,
\begin{equation}\label{eq:vinfty}
 \sup_{0\leq t\leq T}\norm[L^\infty]{v(t)}\leq C_T.
\end{equation}
\end{lemma}

\begin{proof}
Set
\[
 a=\frac{A\gamma}{\alpha}\rho^{\gamma-\delta}
   +\frac{|h|^2}{\alpha}\rho^{2-\delta},
 \qquad G=\rho^2h\cdot k.
\]
Along $X(t;x)$, equation \eqref{eq:vequation} becomes the ordinary
differential equation
\[
 \frac{\dd}{\dd t}V(t)+a(t,X(t;x))V(t)
 =a(t,X(t;x))U(t)-G(t,X(t;x)),
\]
where $V(t)=v(t,X(t;x))$ and $U(t)=u(t,X(t;x))$.  Its integrating-factor
formula is
\begin{align}
 V(t)={}&e^{-\int_0^t a(\tau,X(\tau;x))\,\dd\tau}v_0(x) +\int_0^t e^{-\int_s^t a(\tau,X(\tau;x))\,\dd\tau}
       (aU-G)(s,X(s;x))\,\dd s.
 \label{eq:vvariation}
\end{align}
Because $a\geq0$, both exponential factors are at most one.

Suppose first that $0<\delta<1$.  With $\iota=1/2$,
\begin{equation}\label{eq:au-sublinear}
 a|u|\leq C\left(
 \rho^{\gamma-\delta-1/2}+\rho^{3/2-\delta}
 \right)|\rho^{1/2}u|.
\end{equation}
The two displayed density powers are nonnegative precisely because
$\gamma\geq\delta+1/2$ and $\delta<1$.  Thus the density upper bound and
\eqref{eq:rhouint-sublinear} make $a|u|$ integrable along every path,
uniformly in its starting point.

When $\delta=1$, choose
\begin{equation}\label{eq:iotachoice}
 \frac12<\iota<\min\{1,\gamma-1\}.
\end{equation}
Such a number exists exactly because $\gamma>3/2$.  Now
\begin{equation}\label{eq:au-linear}
 a|u|\leq C\left(
 \rho^{\gamma-1-\iota}+\rho^{1-\iota}
 \right)|\rho^\iota u|,
\end{equation}
and both powers are again nonnegative.  Use
\eqref{eq:rhouint-linear}.

Finally, in both regimes,
\[
 |G|\leq\norm[L^\infty]{\rho}^2
       \norm[L^\infty]{h_0}\norm[L^\infty]{k_0}\leq C_0.
\]
Also $v_0\in L^\infty$ because $u_0\in H^2$ and
$v_0-u_0=(\alpha/\delta)\psi_0$ with $\psi_0\in H^1$.  Taking the supremum
over starting points in \eqref{eq:vvariation} proves
\eqref{eq:vinfty}.  The calculation also shows transparently why the two
different lower bounds on $\gamma$ occur in the two main theorems.
\end{proof}

\section{High-order closure}\label{sec:higher}

We now isolate every estimate where the Lorentz force changes the
degenerate Navier--Stokes argument.  This is essential: an unweighted
$b\in H^2$ assumption alone would not control $b\cdot b_x/\rho$ at vacuum.

\begin{lemma}[Source estimates without derivative loss]\label{lem:source}
Let $\mathcal M$ be defined by \eqref{eq:Mrepresentation} and put
\begin{equation}\label{eq:HG}
 H=\rho|h|^2,\qquad G=\rho^2h\cdot k.
\end{equation}
Then
\begin{equation}\label{eq:Mdecomp}
 \mathcal M=H_x-G,\qquad \Dt G=-2u_xG.
\end{equation}
Before $\rho_{xx}$ is known, the time-differentiated momentum estimate can
use
\begin{equation}\label{eq:Mttrick}
 \int_\R\mathcal M_tu_t\,\dd x
 =-\int_\R H_tu_{tx}\,\dd x-\int_\R G_tu_t\,\dd x.
\end{equation}
On every interval on which the low-order fluid estimates and
$\norm[H^2]{k}$ are finite,
\begin{equation}\label{eq:HGtime}
 \norm[L^2]{H_t}+\norm[L^2]{G_t}
 \leq C_T(1+\norm[H^1]{u}),
\end{equation}
where terms involving $\rho_x$ may equivalently be expressed through
$\rho_x=\alpha^{-1}\rho^{2-\delta}(v-u)$.
After $\psi_x\in L^2$ is known,
\begin{equation}\label{eq:Mxpsi}
 \norm[L^2]{\mathcal M_x}
 \leq C_T\bigl(1+\norm[L^2]{\psi_x}+\norm[H^1]{k}\bigr).
\end{equation}
\end{lemma}

\begin{proof}
Since $h_x=\rho k$,
\[
 H_x=|h|^2\rho_x+2\rho^2h\cdot k=\mathcal M+G,
\]
which proves the first identity.  Material differentiation and
\eqref{eq:transporthk} give $\Dt G=-2u_xG$.  Integration by parts proves
\eqref{eq:Mttrick}.  Directly,
\begin{align}
 H_t={}&-|h|^2u\rho_x-|h|^2\rho u_x
       -2u\rho^2h\cdot k,\label{eq:Ht}\\
 G_t={}&-2u\rho(h\cdot k)\rho_x-2\rho^2(h\cdot k)u_x
       -u\rho^3|k|^2-u\rho^2h\cdot k_x.\label{eq:Gt}
\end{align}
The one-dimensional product inequality, \cref{lem:kH2,lem:vinfty}, and the
energy bounds give \eqref{eq:HGtime}.  Finally,
\begin{equation}\label{eq:rhoxxpsi}
 \rho_{xx}=
 \frac1\delta\rho^{2-\delta}\psi_x
 +\frac{2-\delta}{\delta^2}\rho^{3-2\delta}\psi^2
\end{equation}
(which also gives $\rho_{xx}=\rho(\psi_x+\psi^2)$ when $\delta=1$).
Substitution in \eqref{eq:Mxrepresentation} proves \eqref{eq:Mxpsi}.
\end{proof}

\begin{proposition}[Finite-time a priori estimate]\label{prop:apriori}
Let $(\rho,u,b)$ be a regular solution on $[0,T]$ satisfying the assumptions
of \cref{thm:global-sublinear} or \cref{thm:global-linear}.  Then
\begin{align}
 &\sup_{0\leq t\leq T}\bigl(
 \norm[H^2]{\rho^{\gamma-1}}+\norm[H^1]{\psi}
 +\norm[H^2]{u}+\norm[L^2]{u_t}
 +\norm[H^2]{k}+\norm[H^2]{b}\bigr)\nonumber\\
 &\quad+\norm[L^2(0,T;H^3)]{u}
 +\norm[L^2(0,T;H^1)]{u_t}
 \leq C_T.\label{eq:apriori}
\end{align}
For $0<\delta<1$, the weighted quantities in
\eqref{eq:regvelocity} obey the same type of bound, and the quantities in
\eqref{eq:timeweighted-sublinear} are bounded by $C_T$.  For $\delta=1$,
the quantities in \eqref{eq:timeweighted-linear} are bounded by $C_T$.
\end{proposition}

\begin{proof}
We begin by deriving the equations used in the estimates.  Since
$\phi=C_\gamma\rho^{\gamma-1}$ with
$C_\gamma=A\gamma/(\gamma-1)$, equation \eqref{eq:materialrho} gives
$\Dt\phi=-(\gamma-1)\phi u_x$.  Next,
$\psi=\delta\rho^{\delta-2}\rho_x$ and \eqref{eq:Dtrhox} give
\[
 \Dt\psi=-\delta\psi u_x-\delta\rho^{\delta-1}u_{xx}.
\]
Finally, divide the momentum equation by $\rho$.  The identities
\[
 \rho^{-1}(A\rho^\gamma)_x=\phi_x,\qquad
 \rho^{-1}(\alpha\rho^\delta u_x)_x
 =\alpha\rho^{\delta-1}u_{xx}+\alpha\psi u_x
\]
and the definition of $\mathcal M$ give the third equation below.  Thus the
exact reformulated system is
\begin{align}
 \phi_t+u\phi_x+(\gamma-1)\phi u_x&=0,\label{eq:phieq}\\
 \psi_t+(u\psi)_x+(\delta-1)\psi u_x
 +\delta\rho^{\delta-1}u_{xx}&=0,\label{eq:psieq}\\
 u_t+uu_x+\phi_x
 &=\alpha\rho^{\delta-1}u_{xx}+\alpha\psi u_x-\mathcal M.
 \label{eq:ueq}
\end{align}
For $\delta=1$, \eqref{eq:psieq} reduces to
$\psi_t+(u\psi)_x+u_{xx}=0$.  The remainder of the proof is ordered so that
each derivative is estimated only after all quantities on its right-hand
side are known.  The purely hydrodynamic product estimates coincide with
\cite[Lemmas 3.8--3.13]{CaoLiZhu2022}; below we record their resulting
inequalities and derive every additional magnetic estimate explicitly.

\emph{Step 1: unweighted and first-order velocity estimates.}
By \cref{lem:Mlow,lem:vinfty},
\begin{equation}\label{eq:Mfirst}
 \sup_{0\leq t\leq T}\norm[L^2]{\mathcal M(t)}\leq C_T.
\end{equation}
For clarity, we write the first two energy calculations.  Expanding the
viscous derivative in momentum gives the equivalent form
\begin{equation}\label{eq:uexpanded}
 u_t+uu_x+A\gamma\rho^{\gamma-2}\rho_x
 =\alpha(\rho^{\delta-1}u_x)_x
  +\alpha\rho^{\delta-2}\rho_xu_x-\mathcal M.
\end{equation}
For $0<\delta<1$, multiply \eqref{eq:uexpanded} by $u$.  With
$q=\rho^{\delta-1}$,
\begin{align}
 \frac12\frac{\dd}{\dd t}\norm[L^2]{u}^2
 +\alpha\norm[L^2]{q^{1/2}u_x}^2
 ={}&-\int u^2u_x
 -A\gamma\int\rho^{\gamma-2}\rho_xu
 +\alpha\int\rho^{\delta-2}\rho_xu_xu
 -\int\mathcal M u.
 \label{eq:uL2identity}
\end{align}
The first integral is zero.  Since
\begin{equation}\label{eq:rhox-vu}
 \alpha\rho^{\delta-2}\rho_x=v-u,
\end{equation}
the pressure integral becomes
$-(A\gamma/\alpha)\int\rho^{\gamma-\delta}(v-u)u$.
The density powers are bounded by \eqref{eq:rhoinfty}, while $v$ is bounded
by \cref{lem:vinfty}.  Also, because $\delta-1<0$,
\begin{equation}\label{eq:qcoercive}
 q=\rho^{\delta-1}\geq
 \norm[L^\infty]{\rho}^{\delta-1}>0.
\end{equation}
This lower bound comes from an upper bound for $\rho$; it does not assume
that the density stays away from vacuum.  Using
\eqref{eq:rhox-vu}, \eqref{eq:qcoercive}, and Young's inequality, the last
three integrals in \eqref{eq:uL2identity} satisfy
\begin{align}
 \left|A\gamma\int\rho^{\gamma-2}\rho_xu\right|
 &\leq C_T(1+\norm[L^2]{u}^2),\label{eq:uL2pressure}\\
 \left|\alpha\int\rho^{\delta-2}\rho_xu_xu\right|
 &\leq\frac\alpha4\norm[L^2]{q^{1/2}u_x}^2
      +C_T\norm[L^2]{u}^2,\label{eq:uL2visccross}\\
 |\ip{\mathcal M}{u}|
 &\leq C_T(1+\norm[L^2]{u}^2).
 \label{eq:newlow1}
\end{align}
Gronwall's lemma yields the zeroth-order part of \eqref{eq:firstorder}.

For the first derivative, multiply \eqref{eq:uexpanded} by $u_t$ and
integrate the principal viscous term by parts.  One obtains
\begin{align}
 \frac\alpha2\frac{\dd}{\dd t}\int q u_x^2\,\dd x
 +\int u_t^2\,\dd x
 ={}&\frac\alpha2\int q_tu_x^2\,\dd x
 -\int uu_xu_t\,\dd x
 -A\gamma\int\rho^{\gamma-2}\rho_xu_t\,\dd x\nonumber\\
 &+\alpha\int\rho^{\delta-2}\rho_xu_xu_t\,\dd x
 -\int\mathcal M u_t\,\dd x.
 \label{eq:utfirstidentity}
\end{align}
The last term is genuinely new, but \eqref{eq:Mfirst} gives
\begin{equation}\label{eq:newlow2}
 |\ip{\mathcal M}{u_t}|
 \leq\frac18\norm[L^2]{u_t}^2+C_T.
\end{equation}
For reference, the remaining four estimates are
\begin{align}
 \frac\alpha2\left|\int q_tu_x^2\right|
 &\leq C_T\bigl(1+\norm[L^2]{q^{1/2}u_x}^{10/3}\bigr)
       +\frac18\norm[L^2]{u_t}^2,\label{eq:qtestimate}\\
 \left|\int uu_xu_t\right|
 &\leq C_T\norm[L^2]{q^{1/2}u_x}^{3}
       +\frac18\norm[L^2]{u_t}^2,\label{eq:convutestimate}\\
 \left|A\gamma\int\rho^{\gamma-2}\rho_xu_t\right|
 &\leq C_T+\frac18\norm[L^2]{u_t}^2,\label{eq:pressutestimate}\\
 \left|\alpha\int\rho^{\delta-2}\rho_xu_xu_t\right|
 &\leq C_T(1+\norm[L^2]{u_x}^2)
       \norm[L^2]{q^{1/2}u_x}^2
       +\frac18\norm[L^2]{u_t}^2.\label{eq:crossutestimate}
\end{align}
Here \eqref{eq:qtestimate} follows from
$q_t+uq_x+(\delta-1)qu_x=0$; the other lines use
\eqref{eq:rhox-vu}, the one-dimensional inequality
$\norm[L^\infty]{f}^2\leq2\norm[L^2]{f}\norm[L^2]{f_x}$, and Young's
inequality.  The elliptic form of \eqref{eq:uexpanded} simultaneously gives
\begin{equation}\label{eq:uxxelliptic-low}
 \norm[L^2]{q u_{xx}}
 \leq C_T\bigl(1+\norm[L^2]{u_t}+\norm[L^2]{u_x}^{3/2}\bigr).
\end{equation}
Substitution of \eqref{eq:newlow2}--\eqref{eq:crossutestimate} into
\eqref{eq:utfirstidentity} can now be closed without assuming a small
initial norm.  To see this point explicitly, put
$y=\norm[L^2]{q^{1/2}u_x}^2$.  The zeroth-order estimate gives
$y\in L^1(0,T)$, while \eqref{eq:qcoercive} gives
$\norm[L^2]{u_x}^2\leq C_Ty$.  The largest powers on the right are
therefore
\[
 y^{5/3}=y\,y^{2/3},\qquad y^{3/2}=y\,y^{1/2},\qquad y^2=y\,y.
\]
Their second factors belong to $L^1(0,T)$ because $y\in L^1(0,T)$ and
the time interval is finite.  Thus the differential inequality has the
linear Gronwall form $y'\leq C_T+a(t)y$ with $a\in L^1(0,T)$, after the
$u_t$ terms have been absorbed.  Gronwall's lemma gives, for
$0<\delta<1$,
\begin{align}
 &\sup_{0\leq t\leq T}
 \left(\norm[L^2]{u}^2+
 \norm[L^2]{\rho^{(\delta-1)/2}u_x}^2\right)\nonumber\\
 &\quad+\int_0^T\left(
 \norm[L^2]{\rho^{(\delta-1)/2}u_x}^2
 +\norm[L^2]{u_t}^2+\norm[L^2]{\rho^{\delta-1}u_{xx}}^2
 +\norm[L^\infty]{u_x}\right)\dd t\leq C_T.
 \label{eq:firstorder}
\end{align}
When $\delta=1$, $q=1$ and $q_t=0$.  The same calculation therefore becomes
shorter: all density weights in \eqref{eq:firstorder} are absent and the
first term on the right of \eqref{eq:utfirstidentity} vanishes.  The elliptic
form of \eqref{eq:ueq}, \eqref{eq:Mfirst}, and the one-dimensional
interpolation inequality also imply
\begin{equation}\label{eq:uxxintegral}
 \int_0^T\norm[L^2]{u_{xx}}^2\,\dd t\leq C_T.
\end{equation}
No pointwise-in-time bound for $u_t$ is asserted at this stage.  Equations
\eqref{eq:firstorder}--\eqref{eq:uxxintegral} and \cref{lem:kH2} now yield
\begin{equation}\label{eq:kstrong}
 \sup_{0\leq t\leq T}\norm[H^2]{k(t)}\leq C_T.
\end{equation}

\emph{Step 2: $L^2$ effective velocity and low-order $\psi$.}
Set
\[
 a=\frac{A\gamma}{\alpha}\rho^{\gamma-\delta}
   +\frac{|h|^2}{\alpha}\rho^{2-\delta},
 \qquad G=\rho^2h\cdot k.
\]
Multiplication of \eqref{eq:vequation} by $v$ gives
\begin{equation}\label{eq:v2exact}
 \frac12\frac{\dd}{\dd t}\norm[L^2]{v}^2
 +\int_\R a|v|^2\,\dd x
 =\frac12\int_\R u_x|v|^2\,\dd x
  +\int_\R auv\,\dd x-\int_\R Gv\,\dd x.
\end{equation}
Here we used
$\int uvv_x=-\frac12\int u_xv^2$.  Since $a\in L^\infty$, $a\geq0$,
and $G\in L^2$, Young's inequality gives
\begin{equation}\label{eq:v2diff}
 \frac{\dd}{\dd t}\norm[L^2]{v}^2
 \leq C_T(1+\norm[L^\infty]{u_x})\norm[L^2]{v}^2
      +C_T(1+\norm[L^2]{u}^2).
\end{equation}
The coefficient $\norm[L^\infty]{u_x}$ is integrable in time by
\eqref{eq:firstorder}, so Gronwall controls $v$ in $L^2$.  Since
$v-u=(\alpha/\delta)\psi$, this and the $L^2$ estimate for $u$ control
$\psi$ in $L^2$.  Moreover, \eqref{eq:qcoercive} (or $q=1$) gives
$u\in L^\infty(0,T;H^1)$ and hence $u\in L^\infty((0,T)\times\R)$.
Together with \eqref{eq:vinfty}, interpolation
$\norm[L^4]{\psi}^4\leq\norm[L^\infty]{\psi}^2
\norm[L^2]{\psi}^2$ proves
\begin{equation}\label{eq:v2psi4}
 \sup_{0\leq t\leq T}
 \left(\norm[L^2]{v}+\norm[L^2]{\psi}+\norm[L^4]{\psi}\right)\leq C_T,
\end{equation}
because $v-u=(\alpha/\delta)\psi$.

\emph{Step 3: pointwise $u_t$ and second derivatives.}
Differentiate \eqref{eq:ueq} in time and test by $u_t$.  The only term not
present in the hydrodynamic identity is $-\int\mathcal M_tu_t$.  Before
$\rho_{xx}$ is available, use \eqref{eq:Mttrick}.  By
\eqref{eq:Ht}--\eqref{eq:Gt}, \eqref{eq:firstorder},
\eqref{eq:kstrong}, and \eqref{eq:v2psi4},
\begin{align*}
 \left|\int H_tu_{tx}\,\dd x\right|
 &\leq\norm[L^2]{\rho^{(1-\delta)/2}H_t}
       \norm[L^2]{\rho^{(\delta-1)/2}u_{tx}},\\
 \left|\int G_tu_t\,\dd x\right|
 &\leq\norm[L^2]{G_t}\norm[L^2]{u_t}.
\end{align*}
The density upper bound controls the dual factor
$\rho^{(1-\delta)/2}$ in the first line.  Young's inequality therefore
gives
\begin{equation}\label{eq:Mtcontrol}
 \left|\int_\R\mathcal M_tu_t\,\dd x\right|
 \leq\varepsilon\norm[L^2]{\rho^{(\delta-1)/2}u_{tx}}^2
 +C_{\varepsilon,T}\bigl(1+\norm[L^2]{u_t}^2\bigr)
\end{equation}
for $0<\delta<1$; omit the weight for $\delta=1$.  Notice that the bounded
object is the \emph{dual} factor $\rho^{(1-\delta)/2}$, not the singular
weight $\rho^{(\delta-1)/2}$ itself.

For completeness, the remaining terms lead to
\begin{align}
 &\frac{\dd}{\dd t}\norm[L^2]{u_t}^2
 +\norm[L^2]{\rho^{(\delta-1)/2}u_{tx}}^2\leq C_T\left[
 1+\bigl(1+\norm[L^2]{u_{xx}}^2\bigr)\norm[L^2]{u_t}^2
 +\norm[L^2]{\rho^{\delta-1}u_{xx}}^2\right].
 \label{eq:ut-differential}
\end{align}
To see where this comes from, the differentiated convection is
$u_tu_x+uu_{tx}$, the pressure term is integrated by parts as
$\int\phi_tu_{tx}$, and the differentiated $\psi u_x$ term is estimated
using \eqref{eq:v2psi4}.  The one-dimensional interpolation inequality and
Young's inequality absorb one quarter of the weighted $u_{tx}$ norm.  These
are exactly the hydrodynamic products in
\cite[Lemma 3.11]{CaoLiZhu2022}; \eqref{eq:Mtcontrol} shows that the
Lorentz force fits into the same differential inequality without using
$\rho_{xx}$.

At the initial time,
\begin{equation}\label{eq:utinitial}
 u_t(0)=-u_0u_{0x}-\phi_{0x}
 +\alpha\rho_0^{\delta-1}u_{0xx}
 +\alpha\psi_0u_{0x}-\mathcal M_0.
\end{equation}
In the sublinear regime, \eqref{eq:datadelta} gives
$\alpha\rho_0^{\delta-1}u_{0xx}=g_2$, while the factorization of $u_{0x}$
there makes $\psi_0u_{0x}$ square integrable.  In the $\delta=1$ regime,
the terms are controlled directly by $u_0\in H^2$ and
$(\log\rho_0)_x\in H^1$.  Moreover,
$\mathcal M_0\in L^2$ follows from \cref{lem:Mlow} (or directly from
\eqref{eq:Mrepresentation}).  Hence
$\limsup_{\tau\downarrow0}\norm[L^2]{u_t(\tau)}\leq C$.

Integrate \eqref{eq:ut-differential} over $(\tau,t)$, let
$\tau\downarrow0$, and use \eqref{eq:firstorder} to see that every
coefficient on its right is time integrable.  Gronwall's lemma, followed by
the elliptic estimate \eqref{eq:uxxelliptic-low}, yields
\begin{align}
 &\sup_{0\leq t\leq T}\left(
 \norm[L^2]{u_t}^2+\norm[L^2]{u_{xx}}^2
 +\norm[L^2]{\rho^{\delta-1}u_{xx}}^2\right)
+\int_0^T
 \norm[L^2]{\rho^{(\delta-1)/2}u_{tx}}^2\,\dd t\leq C_T.
 \label{eq:secondorder}
\end{align}
Again the weights disappear when $\delta=1$.

\emph{Step 4: the coupled $(\phi_{xx},\psi_x,u_{xxx})$ estimate.}
Writing $q=\rho^{\delta-1}$, differentiation of the elliptic identity
\eqref{eq:ueq} gives
\begin{equation}\label{eq:uxxxexact}
 \alpha q u_{xxx}
 =u_{tx}+(uu_x)_x+\phi_{xx}
 -\alpha q_xu_{xx}-\alpha(\psi u_x)_x+\mathcal M_x,
 \qquad q_x=\frac{\delta-1}{\delta}\psi.
\end{equation}
We explain the absorption in this elliptic estimate.  First,
\[
 (\psi u_x)_x=\psi_xu_x+\psi u_{xx},\qquad
 (uu_x)_x=u_x^2+uu_{xx}.
\]
The terms containing $\psi_x$ are bounded by
$\norm[L^2]{\psi_x}\norm[L^\infty]{u_x}$.  The term with no known
$L^\infty$ bound for $\psi$ is handled by interpolation:
\begin{align}
 \norm[L^2]{\psi u_{xx}}
 &\leq\norm[L^4]{\psi}\norm[L^4]{u_{xx}}\nonumber\\
 &\leq C\norm[L^4]{\psi}\norm[L^2]{u_{xx}}^{3/4}
       \norm[L^2]{u_{xxx}}^{1/4}
 \leq\varepsilon\norm[L^2]{u_{xxx}}+C_{\varepsilon,T}.
 \label{eq:psiuxxabsorb}
\end{align}
Because $q_x=((\delta-1)/\delta)\psi$, the same estimate treats
$q_xu_{xx}$.  The convection terms are easier, since Step~3 controls
$u$ in $H^2$.  Finally, $q$ has the lower bound
\eqref{eq:qcoercive} when $0<\delta<1$ and equals one when $\delta=1$.
Choose $\varepsilon$ small enough to move the last term in
\eqref{eq:psiuxxabsorb} to the left.  The additional force
$\mathcal M_x$ is controlled by \eqref{eq:Mxpsi}; hence
\begin{align}
 &\norm[L^2]{u_{xxx}}+
 \norm[L^2]{\rho^{\delta-1}u_{xxx}}
\leq C_T\left(1+\norm[L^2]{u_{tx}}
 +\norm[L^2]{\phi_{xx}}+\norm[L^2]{\psi_x}\right).
 \label{eq:ellipticthird}
\end{align}
For $0<\delta<1$ one also obtains the intermediate
$\rho^{(\delta-1)/2}u_{xxx}$ bound.  Applying $\partial_x^2$ to
\eqref{eq:phieq} and $\partial_x$ to \eqref{eq:psieq}, and testing by
$\phi_{xx}$ and $\psi_x$, respectively, gives
\begin{align}
 \frac12\frac{\dd}{\dd t}\norm[L^2]{\phi_{xx}}^2
 ={}&-\left(\gamma+\frac12\right)
       \int_\R u_x|\phi_{xx}|^2\,\dd x
 -(2\gamma-1)\int_\R u_{xx}\phi_x\phi_{xx}\,\dd x
 -(\gamma-1)\int_\R\phi u_{xxx}\phi_{xx}\,\dd x,
 \label{eq:phixxidentity}\\
 \frac12\frac{\dd}{\dd t}\norm[L^2]{\psi_x}^2
 ={}&-\left(\delta+\frac12\right)
       \int_\R u_x|\psi_x|^2\,\dd x
 -(2\delta-1)\int_\R\psi u_{xx}\psi_x\,\dd x
 -\delta\int_\R\rho^{\delta-1}u_{xxx}\psi_x\,\dd x.
 \label{eq:psixidentity}
\end{align}
These identities follow only from the product rule and one integration by
parts.  For example,
\[
 \left|\int\phi u_{xxx}\phi_{xx}\right|
 \leq C_T\bigl(\norm[L^2]{u_{xxx}}^2
                +\norm[L^2]{\phi_{xx}}^2\bigr),
\]
and
\[
 \left|\int\psi u_{xx}\psi_x\right|
 \leq\norm[L^4]{\psi}\norm[L^4]{u_{xx}}
       \norm[L^2]{\psi_x}
 \leq\varepsilon\norm[L^2]{u_{xxx}}^2
   +C_{\varepsilon,T}(1+\norm[L^2]{\psi_x}^2).
\]
The remaining products are estimated in the same way, using the already
known $H^2$ bound for $u$ and the $L^4$ bound for $\psi$.  Consequently,
Consequently,
\begin{align}
 \frac{\dd}{\dd t}\norm[L^2]{\phi_{xx}}^2
 &\leq C_T\left(1+\norm[L^2]{\phi_{xx}}^2+
                    \norm[L^2]{u_{xxx}}^2\right),\label{eq:phixxenergy}\\
 \frac{\dd}{\dd t}\norm[L^2]{\psi_x}^2
 &\leq C_T\left(1+\norm[L^2]{\psi_x}^2+
                    \norm[L^2]{\rho^{\delta-1}u_{xxx}}^2\right).
 \label{eq:psixenergy}
\end{align}
Substitute \eqref{eq:ellipticthird} into
\eqref{eq:phixxenergy}--\eqref{eq:psixenergy}.  The resulting right-hand
side contains only the two unknown energies
$\norm[L^2]{\phi_{xx}}^2+\norm[L^2]{\psi_x}^2$, an integrable forcing
$\norm[L^2]{u_{tx}}^2$, and constants already controlled.  This ordering is
what removes the apparent circularity between $u_{xxx}$ and $\psi_x$.
Integrating the $u_{tx}$ term with \eqref{eq:secondorder} and applying
Gronwall yield
\begin{align}
 &\sup_{0\leq t\leq T}
 \left(\norm[L^2]{\phi_{xx}}^2+\norm[L^2]{\psi_x}^2\right)
+\int_0^T\left(\norm[L^2]{u_{xxx}}^2+
 \norm[L^2]{\rho^{\delta-1}u_{xxx}}^2\right)\dd t\leq C_T.
 \label{eq:highdensity}
\end{align}
The standard transport estimates then give
$\phi_{tx},\psi_t\in L^\infty(0,T;L^2)$ and
$\phi_{tt}\in L^2(0,T;L^2)$.  Formula \eqref{eq:rhoxxpsi} gives
$\rho\in L^\infty(0,T;H^2)$; \cref{lem:force} gives
$b\in L^\infty(0,T;H^2)$.

\emph{Step 5: time weights.}
Only now do we estimate $\mathcal M_t$ directly.  Step~4 has supplied
$\psi_x$ and therefore $\rho_{xx}$ through \eqref{eq:rhoxxpsi}; attempting
this direct estimate in Step~3 would have been circular.  Differentiation of
\eqref{eq:Mrepresentation} gives
\begin{align}
 \mathcal M_t={}&2h\cdot h_t\,\rho_x+|h|^2\rho_{xt}
 +2\rho\rho_t h\cdot k+\rho^2h_t\cdot k
 +\rho^2h\cdot k_t,\label{eq:Mtexact}\\
 \rho_{xt}={}&-u\rho_{xx}-2u_x\rho_x-\rho u_{xx},
 \qquad h_t=-u\rho k,\qquad k_t=-uk_x.
 \label{eq:sourcetime}
\end{align}
Consequently,
\begin{equation}\label{eq:Mtstrong}
 \sup_{0\leq t\leq T}\norm[L^2]{\mathcal M_t(t)}\leq C_T.
\end{equation}
For instance,
$\norm[L^2]{|h|^2\rho_{xt}}$ is controlled by
\eqref{eq:sourcetime}, the $H^2$ bounds for $u$ and $\rho$, and the density
upper bound.  The factors $h_t=-u\rho k$ and $k_t=-uk_x$ are controlled by
$u\in H^2$, $h\in L^\infty$, and $k\in H^2$.  The remaining products in
\eqref{eq:Mtexact} have the same or lower order.

Differentiate \eqref{eq:ueq} in time and test by $u_{tt}$.  The actual new
term satisfies
\begin{equation}\label{eq:Mutt}
 |\ip{\mathcal M_t}{u_{tt}}|
 \leq\tfrac14\norm[L^2]{u_{tt}}^2+C_T.
\end{equation}
The time-differentiated energy inequality therefore has the form
\begin{equation}\label{eq:uttenergy}
 \frac{\dd}{\dd t}
 \norm[L^2]{\rho^{(\delta-1)/2}u_{tx}}^2
 +\norm[L^2]{u_{tt}}^2
 \leq C_T\left(1+
 \norm[L^2]{\rho^{(\delta-1)/2}u_{tx}}^2
 +\norm[L^2]{\mathcal M_t}^2\right).
\end{equation}
Write
$Y(t)=\norm[L^2]{\rho^{(\delta-1)/2}u_{tx}(t)}^2$ and let $Z$ denote the
$u_{tt}$ dissipation.  Multiply \eqref{eq:uttenergy} by $t$ and integrate.
The lower endpoint needs a short limiting argument because $Y(0)$ need not
be defined.  Since $Y\in L^1(0,T)$ by \eqref{eq:secondorder}, there is a
sequence $\tau_j\downarrow0$ such that
$\tau_jY(\tau_j)\to0$.  Integrating over $(\tau_j,t)$ gives
\[
 \int_{\tau_j}^t sY'(s)\,\dd s
 =tY(t)-\tau_jY(\tau_j)-\int_{\tau_j}^tY(s)\,\dd s.
\]
Letting $j\to\infty$ is the reason no value of $u_{tx}(0)$ is required.
The unweighted integral of $Y$ is already finite, and
\eqref{eq:Mtstrong} controls the remaining source.  Gronwall's lemma
therefore gives
\begin{align}
 &t\norm[L^2]{\rho^{(\delta-1)/2}u_{tx}(t)}^2
+\int_0^t s\left(
 \norm[L^2]{u_{tt}}^2+\norm[L^2]{u_{txx}}^2\right)\dd s\leq C_T.
 \label{eq:timeweightcorrect}
\end{align}
To recover $u_{txx}$, differentiate the elliptic momentum identity in time:
\begin{align}
 \alpha\rho^{\delta-1}u_{txx}
 ={}&u_{tt}+(uu_x)_t+\phi_{tx}
 -\alpha(\rho^{\delta-1})_tu_{xx}
 -\alpha\psi_tu_x-\alpha\psi u_{tx}+\mathcal M_t.
 \label{eq:utxxexact}
\end{align}
The only additional right-hand side is $\mathcal M_t$, controlled by
\eqref{eq:Mtstrong}; every other factor has already been estimated in
Steps~3--4.  Since $\rho^{\delta-1}$ has the positive lower bound
\eqref{eq:qcoercive} for $0<\delta<1$, \eqref{eq:utxxexact} gives the
unweighted $u_{txx}$ norm appearing in \eqref{eq:timeweightcorrect}.  For
$\delta=1$ the coefficient is simply one.  The weighted
spatial estimates in \eqref{eq:regvelocity} follow from the same elliptic
identities; the magnetic factor
\begin{equation}\label{eq:weightedM}
 \rho^{1-\delta}\mathcal M
 =\frac{|h|^2}{\delta}\rho^{3-2\delta}\psi
  +\rho^{3-\delta}h\cdot k
\end{equation}
contains no negative power of $\rho$.  In fact, differentiating this
identity, using $h_x=\rho k$, and applying the density upper bound gives
\begin{equation}\label{eq:weightedMH1}
 \norm[H^1]{\rho^{1-\delta}\mathcal M}
 \leq C_T\bigl(1+\norm[H^1]{\psi}+\norm[H^1]{k}\bigr),
\end{equation}
because $3-2\delta\geq1$ and $3-\delta\geq2$.  Thus the weighted elliptic
recoveries introduce no singular coefficient.  Combining Steps 1--5 proves
\eqref{eq:apriori}.
\end{proof}

\begin{remark}[The exponent thresholds]\label{rem:exponents}
The magnetic estimates use only $0<\delta\leq1$ and the density upper bound.
The restrictions $\gamma\geq\delta+\tfrac12$ and $\gamma>\tfrac32$ enter in
\eqref{eq:vinfty}, exactly where they enter the hydrodynamic proof.  The
magnetic coefficient in \eqref{eq:vequation} is nonnegative and has the more
favorable exponent $2-\delta$.
\end{remark}

\section{Local solvability and continuation}\label{sec:local}

The local construction cannot be obtained by applying the hydrodynamic
theorem verbatim, because $h$ and $k$ must be approximated while preserving
$h_x=\rho k$.  We therefore record the augmented construction.

\begin{proposition}[Local existence and continuation]\label{prop:local}
Under the assumptions of \cref{thm:global-sublinear} or
\cref{thm:global-linear}, there is $T_*>0$ and a unique
regular solution of \eqref{eq:system} on $[0,T_*]$.  Let $T^{\max}$ be its
maximal time.  If $T^{\max}<\infty$ and
\begin{equation}\label{eq:continuationquantity}
 \sup_{t<T^{\max}}\left(
 \norm[H^2]{\rho^{\gamma-1}}+\norm[H^1]{\psi}
 +\norm[H^2]{u}+\norm[L^2]{u_t}
 +\norm[L^\infty]{h}+\norm[H^2]{k}\right)<\infty,
\end{equation}
including the weighted norms in \eqref{eq:regvelocity} when
$0<\delta<1$, then the solution extends beyond $T^{\max}$.
\end{proposition}

\begin{proof}
The singular-parabolic local construction for $(\phi,u,\psi)$ is the one in
\cite{XinZhu2021,CaoLiZhu2022}; for $\delta=1$ one may equivalently use the
shallow-water construction in \cite{LiPanZhu2017}.  We verify that its
positive-background approximation and difference estimates remain uniform
after the two magnetic transport equations are added.  The
constraint-preserving Picard map, the explicit energy spaces, the compactness
argument, and the complete augmented difference estimate are recorded in
\ref{app:local-construction}.

\emph{Step 1: compatible positive-background data on $\R$.}
Put $C_\gamma=A\gamma/(\gamma-1)$ and, for $\eta\in(0,1)$, define
\begin{equation}\label{eq:etaapprox}
 \phi_0^\eta=\phi_0+\eta,\qquad
 \rho_0^\eta=(C_\gamma^{-1}\phi_0^\eta)^{1/(\gamma-1)},\qquad
 \psi_0^\eta=\delta(\rho_0^\eta)^{\delta-2}(\rho_0^\eta)_x.
\end{equation}
The fluid compatibility data are approximated exactly as in
\cite[Section 2.1]{CaoLiZhu2022}.  For the magnetic variables set
\begin{equation}\label{eq:magapprox}
 h_0^\eta=h_0,\qquad
 k_0^\eta=\frac{\rho_0}{\rho_0^\eta}k_0.
\end{equation}
Then $(h_0^\eta)_x=\rho_0^\eta k_0^\eta$.  Because $k_0$ is supported in a
fixed compact set on which $\rho_0$ has a positive minimum, composition and
product estimates give
\begin{equation}\label{eq:uniformkapprox}
 \sup_{0<\eta<1}\norm[H^2]{k_0^\eta}\leq C\norm[H^2]{k_0},\qquad
 k_0^\eta\longrightarrow k_0\quad\text{in }H^2.
\end{equation}
Thus no spatial truncation or artificial boundary condition is introduced.

\emph{Step 2: the approximate augmented system.}
For fixed $\eta$, the density has a positive far-field background.  Solve
\eqref{eq:phieq}--\eqref{eq:ueq} on $\R$ together with
\begin{equation}\label{eq:augtransport}
 h_t+uh_x=0,\qquad k_t+uk_x=0.
\end{equation}
For $0<\delta<1$, writing
\[
 e=\frac{\delta-1}{2(\gamma-1)}<0,\qquad
 a=C_\gamma^{(1-\delta)/(\gamma-1)},
\]
one has $\rho^{\delta-1}=a\phi^{2e}$.  Hence
\eqref{eq:phieq}--\eqref{eq:psieq} have exactly the singular-parabolic
principal part used in \cite{XinZhu2021,CaoLiZhu2022}.  The new force is
\begin{equation}\label{eq:Mpsi}
 \mathcal M=\frac{|h|^2}{\delta}\rho^{2-\delta}\psi
              +\rho^2h\cdot k.
\end{equation}
For each fixed $\eta$ it is smooth in the approximate variables.  More
importantly, all density powers in \eqref{eq:Mpsi} are nonnegative, so the
following estimates do not involve $\inf\rho^\eta$:
\begin{align}
 \norm[L^2]{\mathcal M}
 &\leq C_R\bigl(1+\norm[L^2]{\psi}+\norm[L^2]{k}\bigr),\label{eq:Mlocal0}\\
 \norm[H^1]{\mathcal M}
 &\leq C_R\bigl(1+\norm[H^1]{\psi}+\norm[H^2]{k}\bigr),\label{eq:Mlocal1}
\end{align}
whenever
$\norm[H^2]{\phi-\eta}+\norm[H^2]{u}+\norm[L^\infty]{h}\leq R$.
These follow by differentiating \eqref{eq:Mpsi}, using
$h_x=\rho k$ and
$\rho_x=\delta^{-1}\rho^{2-\delta}\psi$, and applying
$H^1(\R)\hookrightarrow L^\infty(\R)$.

Let $\mathcal Y^\eta$ and $\mathcal Z^\eta$ be the explicit functionals
$\mathscr Y_\eta$ and $\mathscr Z_\eta$ in
\eqref{eq:app-Y}--\eqref{eq:app-Z}.  Equations
\eqref{eq:k0energy}--\eqref{eq:k2energy} and
\eqref{eq:Mlocal0}--\eqref{eq:Mlocal1} show that every new term in the
linearized energy is bounded by
\[
 \varepsilon\mathcal Z^\eta+
 C_{\varepsilon}(1+\mathcal Y^\eta)^N,
\]
where $N$ and the constant are independent of $\eta$.  The complete local
energy inequality is therefore
\begin{equation}\label{eq:localODE}
 \frac{\dd}{\dd t}\mathcal Y^\eta+\mathcal Z^\eta
 \leq C(1+\mathcal Y^\eta)^N,\qquad
 \sup_{0<\eta<1}\mathcal Y^\eta(0)<\infty.
\end{equation}
The hydrodynamic terms in \eqref{eq:localODE} are precisely those in the
cited singular-parabolic iteration; \eqref{eq:Mlocal0}--\eqref{eq:Mlocal1}
list all additional terms.  The ODE comparison lemma supplies a time
$T_*>0$, independent of $\eta$, on which
\begin{equation}\label{eq:uniformlocal}
 \sup_{0<\eta<1}\left\{
 \sup_{0\leq t\leq T_*}\mathcal Y^\eta(t)
 +\int_0^{T_*}\mathcal Z^\eta(t)\,\dd t\right\}<\infty.
\end{equation}
The standard regularized transport--parabolic iteration first constructs
the positive-background solution; \eqref{eq:localODE} prevents its norm from
leaving the iteration ball before $T_*$.  The fact that every iterate
preserves $h_x=\rho k$ and the density constraint is proved in
\eqref{eq:app-F-iterate}--\eqref{eq:app-Q-iterate}.

\emph{Step 3: constraints and the vacuum limit.}
The differential constraint
\[
 F=h_x-\rho k
\]
satisfies
\begin{equation}\label{eq:defect}
 \Dt F=-u_xF.
\end{equation}
It vanishes initially by \eqref{eq:magapprox}.  The analogous defects
relating $\phi$, $\psi$, and $\rho$ satisfy the homogeneous equations in
\cite[Section 2]{CaoLiZhu2022}; hence all constraints persist.
The uniform estimate \eqref{eq:uniformlocal}, weak compactness at top order,
and local Aubin--Lions compactness at lower order permit
$\eta\downarrow0$.  The limit satisfies \eqref{eq:phieq}--\eqref{eq:ueq},
\eqref{eq:augtransport}, and the constraints.  Thus $b=\rho h$ solves
$b_t+(ub)_x=0$.  Positivity at every finite point follows from the
characteristic formula for $\phi$ and $\phi_0>0$.  The cutoff argument of
\cite[Lemma 2.2]{CaoLiZhu2022} recovers $\rho\in C([0,T_*];L^1)$.

\emph{Step 4: uniqueness.}
Put
$\vartheta_i=\rho_i^{1-\delta}$, $f_i=(\log\rho_i)_x$, and
$\mathcal N=\vartheta_1\mathcal M_1-\vartheta_2\mathcal M_2$.
Multiplication of \eqref{eq:ueq} by $\vartheta_i$ gives a constant
principal viscous coefficient.  When $0<\delta<1$,
\[
 \vartheta_i\mathcal M_i
 =|h_i|^2\rho_i^{2-\delta}f_i
   +\rho_i^{3-\delta}h_i\cdot k_i.
\]
Density comparability along the two flows and the fact that the powers of
$\vartheta_i$ in this formula exceed $3/2$ give
\[
 \norm[L^2]{\mathcal N/\sqrt{\vartheta_1}}
 \le C_T\bigl(
  \norm[H^1]{\vartheta_1-\vartheta_2}
 +\norm[L^2]{f_1-f_2}
 +\norm[H^1]{h_1-h_2}
 +\norm[L^2]{k_1-k_2}\bigr).
\]
For $\delta=1$, $\vartheta_i=1$, and the corresponding estimate uses
$\|\rho_1-\rho_2\|_{H^1}$ instead.  Thus the only new pairings in the
fluid difference identity, $\langle\mathcal N,u_1-u_2\rangle$ and
$\langle\mathcal N,u_{1t}-u_{2t}\rangle$, are absorbed by the explicit
vacuum-weighted functional in
\eqref{eq:app-fluid-functional}--\eqref{eq:app-fluid-dissipation}.
The transport equations for $h_1-h_2$ and $k_1-k_2$, together with the raw
density difference equation in the case $\delta=1$, then give, for some
coercivity constant $c_0>0$,
\begin{equation}\label{eq:differenceenergy}
 \frac{\dd}{\dd t}\mathfrak E(t)+\frac{c_0}{2}\mathfrak D_f(t)
 \leq C_T\mathfrak E(t).
\end{equation}
All quantities and every Lorentz pairing in this argument are displayed in
\eqref{eq:app-fluid-difference-variables}--\eqref{eq:app-final-uniqueness}.
Since $\mathfrak E(0)=0$, uniqueness follows.

\emph{Step 5: continuation.}
Assume \eqref{eq:continuationquantity}.  For every $T<T^{\max}$, the local
energy estimate has a constant bounded by the same function of
$T^{\max}$ and the displayed continuation norm.  The equations give strong
terminal traces in the lower-order spaces and weakly continuous traces at
top order.  The weighted time estimates preserve the fluid compatibility
conditions; \cref{lem:commute} preserves the magnetic class and
compact support.  The terminal data therefore satisfy the hypotheses of the
uniform local construction, which extends the solution past $T^{\max}$.
\end{proof}

\begin{remark}\label{rem:localdetail}
The proof uses the established singular-parabolic iteration only for its
hydrodynamic principal part.  The independent estimates
\eqref{eq:uniformkapprox}, \eqref{eq:Mlocal0}--\eqref{eq:Mlocal1}, and
\eqref{eq:defect} are what make the augmented magnetic approximation
uniform.
\end{remark}

\section{Proofs of the two main theorems}\label{sec:global}

\begin{proof}[Proof of \cref{thm:global-sublinear}]
Fix $0<\delta<1$.  By \cref{prop:local}, the data generate a unique regular
solution on a nonempty interval.  Let $[0,T^{\max})$ be its maximal interval
of existence.  We prove that $T^{\max}=\infty$ by contradiction.

Assume $T^{\max}<\infty$ and take any $T<T^{\max}$.  Mass and energy are
bounded by \cref{lem:energy}; the density upper bound is independent of
$T$ by \cref{lem:densityupper}; and $h,k$ remain transported by
\cref{lem:commute}.  The modified entropy, the pointwise effective-velocity
estimate, and the high-order estimate then give
\[
 \sup_{0\leq t\leq T}\bigl(
 \norm[H^2]{\rho^{\gamma-1}}+\norm[H^1]{\psi}
 +\norm[H^2]{u}+\norm[L^2]{u_t}
 +\norm[L^\infty]{h}+\norm[H^2]{k}\bigr)\leq C_{T^{\max}}.
\]
The same constant controls the singular vacuum-compatible quantities
$\rho^{(\delta-1)/2}u_x$ and
$\rho^{\delta-1}u_{xx}$ in \eqref{eq:regvelocity}.  The notation
$C_{T^{\max}}$ is legitimate because every estimate obtained above depends
monotonically on the length of the finite interval, not on a positive lower
bound for $\rho$.

Let $T\uparrow T^{\max}$.  The equations supply strong terminal traces one
spatial derivative below the top order, while the uniform bounds supply
weak terminal traces at the top order.  The time-weighted estimates preserve
the two compatibility relations in \eqref{eq:datadelta}.  Equations
\eqref{eq:hkcharacteristics} preserve $k\in H^2$ and its compact support.
Consequently the terminal data satisfy the hypotheses of the local
construction with the same bounds.  The continuation statement in
\cref{prop:local} extends the solution to
$[0,T^{\max}+\varepsilon]$, contradicting maximality.  Thus
$T^{\max}=\infty$.

The identities in the equations give the time continuities stated in the
theorem, and \eqref{eq:timeweightcorrect} gives
\eqref{eq:timeweighted-sublinear}.  The characteristic identities
\eqref{eq:transport-sublinear} are \eqref{eq:hkcharacteristics}.  Finally,
for any $0<\tau<T$ and compact $K\Subset\R$, continuity and pointwise
positivity imply
\[
 \inf_{[\tau,T]\times K}\rho>0.
\]
On this cylinder the velocity equation is uniformly parabolic, so standard
interior regularity upgrades the regular solution to a classical one.  Since
$\tau,T$, and $K$ are arbitrary, the solution is classical on
$(0,\infty)\times\R$.
\end{proof}

\begin{proof}[Proof of \cref{thm:global-linear}]
Now set $\delta=1$.  By \cref{prop:local}, there is again a unique maximal
regular solution on $[0,T^{\max})$.  Suppose for contradiction that
$T^{\max}<\infty$.  For every $T<T^{\max}$,
\cref{lem:commute,lem:energy,lem:densityupper,lem:BD,lem:vinfty} and
\cref{prop:apriori} give
\[
 \sup_{0\leq t\leq T}\left(
 \norm[H^2]{\rho^{\gamma-1}}+\norm[H^1]{(\log\rho)_x}
 +\norm[H^2]{u}+\norm[L^2]{u_t}
 +\norm[L^\infty]{h}+\norm[H^2]{k}\right)\leq C_{T^{\max}}.
\]
There are no singular density weights in this line: the principal coefficient
in the reformulated velocity equation is
$\rho^{\delta-1}=1$.  Hence every continuation quantity in
\eqref{eq:continuationquantity} remains finite uniformly as
$T\uparrow T^{\max}$.

The equations yield the required terminal traces, while
\eqref{eq:hkcharacteristics} preserves the magnetic data class and compact
support.  Restarting the $\delta=1$ local construction at $T^{\max}$ extends
the solution past its maximal time, a contradiction.  Therefore
$T^{\max}=\infty$.

The unweighted time estimate \eqref{eq:timeweighted-linear} follows from
\eqref{eq:timeweightcorrect} with all density weights equal to one, and the
transport conclusion is \eqref{eq:hkcharacteristics}.  On any compact
cylinder $[\tau,T]\times K$ with $\tau>0$, pointwise positivity and
continuity again give a positive density lower bound.  Uniform interior
parabolic regularity then proves that the global regular solution is
classical on $(0,\infty)\times\R$.
\end{proof}

\begin{proof}[Proof of \cref{cor:constant}]
Here $h_0=c$ and $k_0=0$.  By \cref{lem:commute}, $h=c$ for all time, so
$b=c\rho$.  Substitution in momentum gives
\eqref{eq:compoundpressure}.  According to whether $0<\delta<1$ or
$\delta=1$, apply \cref{thm:global-sublinear} or
\cref{thm:global-linear}, respectively.
\end{proof}

\Needspace{8\baselineskip}
\begin{corollary}[Persistence of nonzero momentum]\label{cor:momentum}
The total momentum
\begin{equation}\label{eq:momentum}
 P_0=\int_\R\rho_0u_0\,\dd x
\end{equation}
is conserved.  If $P_0\neq0$, then
\begin{equation}\label{eq:nondecay}
 \norm[L^\infty]{u(t)}\geq\frac{|P_0|}{M_0}
 \qquad\text{for all }t\geq0.
\end{equation}
Thus the velocity cannot converge to zero in $L^\infty$.
\end{corollary}

\begin{proof}
Integrate momentum against cutoffs tending to one.  The pressure, magnetic
pressure, and viscous flux vanish at infinity by the regularity proved above,
so momentum is conserved.  Then
$|P_0|\leq\norm[L^\infty]{u(t)}\int\rho(t)
=M_0\norm[L^\infty]{u(t)}$.
\end{proof}

\section{Scope and open extensions}\label{sec:scope}

The two main theorems use neither small data nor a constant ratio $b/\rho$,
but they do use zero magnetic diffusivity and a controlled mass derivative of the
frozen flux.  If a positive resistivity $\nu b_{xx}$ is added, $h$ and $k$
are no longer transported.  A natural variable in that setting is
\[
 J=\frac{b}{\rho^{(1+\delta)/2}},
\]
which appears in the recent local theory \cite{LiLuYuan2026}; a global
coupled $H^1$--$H^2$ estimate for $J$ would still be needed.

If a nonzero longitudinal magnetic field and transverse velocity are
retained, the induction equation for a weighted magnetic variable contains
a term of the form $\rho^{-(1+\delta)/2}w_x$.  Basic energy does not control
this quantity at far-field vacuum.  Therefore neither
\cref{thm:global-sublinear} nor \cref{thm:global-linear} covers the full
planar MHD system.  A coupled Alfv\'en effective entropy
in that regime remains an open problem.

\appendix
\section{Uniform local construction for the augmented magnetic system}
\label{app:local-construction}

This appendix supplies the details needed in
\cref{prop:local}.  In particular, we specify the magnetic function space,
construct the positive-background approximate solutions without a spatial
truncation, and verify that the magnetic terms preserve the uniform
singular-parabolic estimates of \cite{XinZhu2021,CaoLiZhu2022}.

\subsection{The two-tail magnetic space}

Since $k_0$ is compactly supported and
$(h_0)_x=\rho_0k_0$, there are vectors $c_-,c_+\in\R^m$ such that
\[
 h_0(x)=c_-\quad\text{for }x\ll-1,
 \qquad
 h_0(x)=c_+\quad\text{for }x\gg1.
\]
Fix $\chi\in C^\infty(\R)$ with
$\chi=0$ on $(-\infty,-1]$ and $\chi=1$ on $[1,\infty)$, and set
\begin{equation}\label{eq:app-hbar}
 \bar h(x)=c_-+(c_+-c_-)\chi(x).
\end{equation}
We use the affine space
\begin{equation}\label{eq:app-magnetic-space}
 \mathcal H_{\bar h}^3
 :=\{h:\ h-\bar h\in H^3(\R;\R^m)\},
 \qquad
 \norm[\mathcal H_{\bar h}^3]{h}
 :=\norm[H^3]{h-\bar h}.
\end{equation}
This space permits two distinct nonzero tail values.  It also contains the
constant-ratio case: if $c_-=c_+=c$, one may take $\bar h=c$.

The assumptions of \cref{thm:global-sublinear} or
\cref{thm:global-linear} imply
$h_0\in\mathcal H_{\bar h}^3$.  Indeed,
$(h_0)_x=\rho_0k_0\in H^2$, while $h_0-\bar h$ is compactly supported.
Here $\rho_0\in H^2$ follows from \eqref{eq:psidef} and
\eqref{eq:rhoxxpsi}.  For later use, the equivalent quantity
\begin{equation}\label{eq:app-h-equivalent}
 \norm[L^2]{h-\bar h}+\norm[H^2]{h_x}
\end{equation}
may be used in place of
$\norm[\mathcal H_{\bar h}^3]{h}$.

\subsection{Positive-background data and constraint-preserving iteration}

Let
\[
 C_\gamma=\frac{A\gamma}{\gamma-1},
 \qquad
 r_\eta=\left(C_\gamma^{-1}\eta\right)^{1/(\gamma-1)}.
\]
For $0<\eta<1$, define
\begin{align}
 \phi_0^\eta&=\phi_0+\eta,
 &
 \rho_0^\eta
 &=\left(C_\gamma^{-1}\phi_0^\eta\right)^{1/(\gamma-1)},
 \label{eq:app-fluid-approx}\\
 \psi_0^\eta
 &=\delta(\rho_0^\eta)^{\delta-2}(\rho_0^\eta)_x,
 &
 h_0^\eta&=h_0,
 &
 k_0^\eta&=\frac{\rho_0}{\rho_0^\eta}k_0.
 \label{eq:app-magnetic-approx}
\end{align}
We do not change the initial velocity:
\begin{equation}\label{eq:app-velocity-approx}
 u_0^\eta:=u_0,\qquad
 u_0^\eta\longrightarrow u_0\quad\text{in }H^2.
\end{equation}
This choice preserves the vacuum compatibility bounds.  Indeed, in the
sublinear case define
\begin{equation}\label{eq:app-g-approx}
 g_1^\eta=\left(\frac{\rho_0}{\rho_0^\eta}\right)^{(1-\delta)/2}g_1,
 \qquad
 g_2^\eta=\left(\frac{\rho_0}{\rho_0^\eta}\right)^{1-\delta}g_2.
\end{equation}
Then
\[
 (u_0^\eta)_x=(\rho_0^\eta)^{(1-\delta)/2}g_1^\eta,
 \qquad
 \alpha(u_0^\eta)_{xx}=(\rho_0^\eta)^{1-\delta}g_2^\eta.
\]
Since $0<\rho_0/\rho_0^\eta\leq1$, dominated convergence gives
\begin{equation}\label{eq:app-g-uniform}
 \norm[L^2]{g_i^\eta}\leq\norm[L^2]{g_i},\qquad
 g_i^\eta\longrightarrow g_i\quad\text{in }L^2,\quad i=1,2.
\end{equation}
When $\delta=1$ the weights equal one and no auxiliary $g_i^\eta$ are
needed.  This is the one-dimensional compatibility approximation used in
\cite[Section 2.1]{CaoLiZhu2022}, written here explicitly.

Because
$\operatorname{supp}k_0$ is contained in a fixed compact set on which
$\rho_0$ has a positive minimum,
\begin{equation}\label{eq:app-k0-uniform}
 \sup_{0<\eta<1}\norm[H^2]{k_0^\eta}<\infty,
 \qquad
 k_0^\eta\longrightarrow k_0\quad\text{in }H^2.
\end{equation}
Moreover,
\begin{equation}\label{eq:app-initial-constraint}
 (h_0^\eta)_x=\rho_0^\eta k_0^\eta=(h_0)_x.
\end{equation}
Thus $h_0^\eta-\bar h=h_0-\bar h$ in $H^3$, uniformly in
$\eta$.

We describe a Picard scheme which preserves all differential constraints at
every iterate.  Start with the time-independent function
$u^0(t,x)=u_0^\eta(x)$.  Given $u^n$, first solve, with this velocity,
\begin{equation}\label{eq:app-transport-iterate}
 \left\{
 \begin{aligned}
  &\rho_t^{n+1}+u^n\rho_x^{n+1}
       +\rho^{n+1}u_x^n=0,\\
  &\psi_t^{n+1}+u^n\psi_x^{n+1}
       +\delta u_x^n\psi^{n+1}
       +\delta(\rho^{n+1})^{\delta-1}u_{xx}^n=0,\\
  &h_t^{n+1}+u^n h_x^{n+1}=0,\\
  &k_t^{n+1}+u^n k_x^{n+1}=0,
 \end{aligned}\right.
\end{equation}
with the complete initial condition
\begin{equation}\label{eq:app-transport-initial}
 (\rho^{n+1},\psi^{n+1},h^{n+1},k^{n+1})|_{t=0}
 =(\rho_0^\eta,\psi_0^\eta,h_0^\eta,k_0^\eta),
\end{equation}
and put
\[
 \phi^{n+1}=C_\gamma(\rho^{n+1})^{\gamma-1}.
\]
Next solve the linear parabolic initial-value problem
\begin{align}
 u_t^{n+1}+u^nu_x^{n+1}+\phi_x^{n+1}
 ={}&\alpha(\rho^{n+1})^{\delta-1}u_{xx}^{n+1}
     +\alpha\psi^{n+1}u_x^{n+1}
     -\mathcal M^{n+1},\label{eq:app-u-iterate}\\*
 \mathcal M^{n+1}
 ={}&\frac{|h^{n+1}|^2}{\delta}
       (\rho^{n+1})^{2-\delta}\psi^{n+1}
     +(\rho^{n+1})^2h^{n+1}\cdot k^{n+1},
 \label{eq:app-M-iterate}\\*
 u^{n+1}(0,x)={}&u_0^\eta(x).
 \label{eq:app-u-initial}
\end{align}
For fixed $\eta$, the characteristic formula gives a strictly positive
lower bound for $\rho^{n+1}$ on every finite iteration interval, so
\eqref{eq:app-u-iterate} is a classical linear parabolic equation.
More explicitly, if
\[
 \frac{\dd}{\dd t}X^n(t,a)=u^n(t,X^n(t,a)),
 \qquad X^n(0,a)=a,
\]
then the first equation in \eqref{eq:app-transport-iterate} is an ordinary
differential equation along $X^n$ and gives
\begin{equation}\label{eq:app-rho-characteristic}
 \rho^{n+1}(t,X^n(t,a))
 =\rho_0^\eta(a)
 \exp\left(-\int_0^t u_x^n(s,X^n(s,a))\,\dd s\right)>0.
\end{equation}
This formula is the only place where a positive lower bound is needed to
solve the fixed-$\eta$ parabolic problem.  The uniform estimates below do
not depend on that lower bound.

This ordering is important.  Define
\[
 F^{n+1}=h_x^{n+1}-\rho^{n+1}k^{n+1}.
\]
Since $\rho^{n+1},h^{n+1}$, and $k^{n+1}$ are transported by the
same $u^n$,
\begin{equation}\label{eq:app-F-iterate}
 (\partial_t+u^n\partial_x)F^{n+1}=-u_x^nF^{n+1}.
\end{equation}
Consequently $F^{n+1}=0$ for every $n$.  Similarly, if
\[
 Q^{n+1}
 =\psi^{n+1}
  -\delta(\rho^{n+1})^{\delta-2}\rho_x^{n+1},
\]
then
\begin{equation}\label{eq:app-Q-iterate}
 (\partial_t+u^n\partial_x)Q^{n+1}
 +\delta u_x^nQ^{n+1}=0,
\end{equation}
and hence $Q^{n+1}=0$.  Thus all identities used below, in particular
\begin{equation}\label{eq:app-M-regular}
 \mathcal M^{n+1}
 =|h^{n+1}|^2\rho_x^{n+1}
  +(\rho^{n+1})^2h^{n+1}\cdot k^{n+1},
\end{equation}
hold at the approximate level, not merely after passage to the limit.

\subsection{Uniform local energy}

For a positive-background state, put
\[
 \omega_\eta=\rho^{(\delta-1)/2},
 \qquad
 a_\eta=\rho^{\delta-1};
\]
both weights are equal to one when $\delta=1$.  We use
\begin{align}
 \mathscr Y_\eta(t)
 :={}&
 \norm[H^2]{\phi-\eta}^2
 +\norm[H^1]{\psi}^2
 +\norm[H^2]{u}^2
 +\norm[L^2]{u_t}^2                                      \nonumber\\
 &+\norm[L^2]{\omega_\eta u_x}^2
 +\norm[L^2]{a_\eta u_{xx}}^2
 +\norm[H^2]{k}^2
 +\norm[H^3]{h-\bar h}^2,                     \label{eq:app-Y}\\
 \mathscr Z_\eta(t)
 :={}&
 \norm[H^3]{u}^2
 +\norm[H^1]{u_t}^2
 +\norm[L^2]{\omega_\eta u_{tx}}^2
 +\norm[H^1]{a_\eta u_{xx}}^2.                \label{eq:app-Z}
\end{align}
The spatial derivatives of $\rho$ are derived uniformly from
\begin{align}
 \rho_x&=\frac1\delta\rho^{2-\delta}\psi,\label{eq:app-rhox}\\
 \rho_{xx}
 &=\frac1\delta\rho^{2-\delta}\psi_x
   +\frac{2-\delta}{\delta^2}
       \rho^{3-2\delta}\psi^2.                 \label{eq:app-rhoxx}
\end{align}
On every compact interval, the zeroth-order density norm follows from the
upper bound.  No global, $\eta$-uniform $L^2$ bound for
$\rho-r_\eta$ is needed in the vacuum limit.

We record the magnetic estimates used in the energy argument.  On a ball
$\mathscr Y_\eta\le R$, with $R$ independent of $\eta$,
\begin{align}
 \norm[H^1]{\mathcal M}
 &\le C_R\left(1+\norm[H^1]{\psi}
                    +\norm[H^2]{k}
                    +\norm[H^3]{h-\bar h}\right),
 \label{eq:app-MH1}\\
 \norm[L^2]{\mathcal M_t}
 &\le C_R(1+\norm[H^2]{u})
       \left(1+\norm[H^1]{\psi}
                    +\norm[H^2]{k}
                    +\norm[H^3]{h-\bar h}\right)^2.
 \label{eq:app-Mt}
\end{align}
The constants depend on an upper bound for $\rho$, but not on its positive
lower bound.  To prove \eqref{eq:app-MH1}, differentiate
\[
 \mathcal M
 =\frac{|h|^2}{\delta}\rho^{2-\delta}\psi
  +\rho^2h\cdot k
\]
and use $h_x=\rho k$, \eqref{eq:app-rhox}, and
$H^1(\R)\hookrightarrow L^\infty(\R)$.  Every resulting density power is
nonnegative.  For \eqref{eq:app-Mt}, before the second density derivative
is available, use instead
\[
 \mathcal M=(\rho|h|^2)_x-\rho^2h\cdot k
\]
and integrate the first term by parts; after \eqref{eq:app-rhoxx}, direct
differentiation in time requires only the norms displayed in
\eqref{eq:app-Y}.

The $H^2$ transport estimate for $k$ is
\begin{equation}\label{eq:app-k-energy}
 \frac{\dd}{\dd t}\norm[H^2]{k}^2
 \le C\left(\norm[L^\infty]{u_x}
             +\norm[L^2]{u_{xx}}\right)\norm[H^2]{k}^2.
\end{equation}
Writing $\widetilde h=h-\bar h$, one has
\[
 \widetilde h_t+u\widetilde h_x=-u\bar h_x.
\]
Either its $H^3$ transport estimate or the equivalent norm
\eqref{eq:app-h-equivalent}, together with
$h_x=\rho k$, gives
\begin{equation}\label{eq:app-h-energy}
 \frac{\dd}{\dd t}\norm[H^3]{\widetilde h}^2
 \le \varepsilon\norm[H^3]{u}^2
      +C_{\varepsilon,R}
       \bigl(1+\norm[H^3]{\widetilde h}^2
                +\norm[H^2]{k}^2\bigr).
\end{equation}

Apply the singular-parabolic estimates of
\cite[Section 3.6]{XinZhu2021} to
\eqref{eq:app-transport-iterate}--\eqref{eq:app-u-iterate}.
The undifferentiated, $x$-differentiated, and $t$-differentiated
momentum estimates require, respectively,
$\mathcal M\in L^2$, $\mathcal M_x\in L^2$, and the dual form of
$\mathcal M_t$ from \eqref{eq:Mttrick}; after the spatial estimate one uses
\eqref{eq:app-Mt}.  More explicitly, the new pairings and the new term in
the elliptic recovery obey the following bounds.  With
$H=\rho|h|^2$ and $G=\rho^2h\cdot k$,
\begin{equation}\label{eq:app-HtGt}
 \norm[L^2]{\omega_\eta^{-1}H_t}+\norm[L^2]{G_t}
 \le C_R(1+\mathscr Y_\eta)^{N/2}.
\end{equation}
Consequently,
\begin{align}
 |\ip{\mathcal M}{u}|+|\ip{\mathcal M}{u_t}|
 &\le \varepsilon\norm[L^2]{u_t}^2
       +C_{\varepsilon,R}(1+\mathscr Y_\eta)^N,
 \label{eq:app-M-low-pairings}\\
 |\ip{\mathcal M_t}{u_t}|
 &=\left|-\ip{(\rho|h|^2)_t}{u_{tx}}
           -\ip{(\rho^2h\cdot k)_t}{u_t}\right| \le \varepsilon\norm[L^2]{\omega_\eta u_{tx}}^2
       +C_{\varepsilon,R}(1+\mathscr Y_\eta)^N,
 \label{eq:app-Mt-dual-pairing}\\
 \norm[L^2]{\mathcal M_x}^2
 &\le C_R(1+\mathscr Y_\eta)^N,
 \label{eq:app-Mx-elliptic}\\
 |\ip{\mathcal M_t}{u_{tt}}|
 &\le \frac14\norm[L^2]{u_{tt}}^2
       +C_R(1+\mathscr Y_\eta)^N.
 \label{eq:app-Mt-time-pairing}
\end{align}
In \eqref{eq:app-Mt-dual-pairing} the first term is integrated by parts in
$x$ before $\rho_{xx}$ is known; the density factor in
$(\rho|h|^2)_t$ supplies the dual weight $\omega_\eta^{-1}$.
Estimate \eqref{eq:app-Mx-elliptic} is inserted in the elliptic formula for
$u_{xxx}$, not paired with a second derivative of $u$.  After that spatial
estimate, \eqref{eq:app-Mt} gives
\eqref{eq:app-Mt-time-pairing}.  These are all Lorentz-force contributions
to the cited hierarchy.  Combining them with
\eqref{eq:app-k-energy}--\eqref{eq:app-h-energy} gives
\begin{align}
 &\sup_{0\le s\le t}\mathscr Y_\eta(s)
   +\int_0^t\mathscr Z_\eta(s)\,\dd s                 
   +\sup_{0<s\le t}
       s\norm[L^2]{\omega_\eta u_{tx}(s)}^2
   +\int_0^t s\left(
       \norm[L^2]{u_{tt}}^2+
       \norm[L^2]{u_{txx}}^2\right)\dd s       \nonumber\\
 &\le C_0+C_R\int_0^t
       \bigl(1+\mathscr Y_\eta(s)\bigr)^N\,\dd s .
 \label{eq:app-uniform-energy}
\end{align}
Here $C_0,R,N$ are independent of $\eta$.  The fluid compatibility
approximation gives uniform initial bounds for all fluid terms in
$\mathscr Y_\eta(0)$.  The magnetic part is uniform by
\eqref{eq:app-k0-uniform}--\eqref{eq:app-initial-constraint};
in particular, the extra contribution to $u_t^\eta(0)$ is uniformly
bounded by \eqref{eq:app-MH1}.

We make the ODE comparison explicit.  Choose
$R_*=2(C_0+1)$ and set
\[
 T_*:=\min\left\{1,
 \frac{R_*-C_0}{C_{R_*}(1+R_*)^N}\right\}.
\]
Suppose, as a bootstrap assumption, that
$\sup_{s\leq t}\mathscr Y_\eta(s)\leq R_*$ for $t\leq T_*$.  Then
\eqref{eq:app-uniform-energy} improves this bound to
$\sup_{s\leq t}\mathscr Y_\eta(s)<R_*$.  Continuity prevents the norm from
being the first to reach $R_*$, so $T_*$ is independent of $\eta$ and
\begin{equation}\label{eq:app-uniform-bound}
 \sup_{0<\eta<1}\left[
  \sup_{0\le t\le T_*}\mathscr Y_\eta(t)
  +\int_0^{T_*}\mathscr Z_\eta(t)\,\dd t\right]<\infty.
\end{equation}

For completeness, we now verify the lower-order contraction that constructs
the positive-background solution for each fixed $\eta$.  Write
\[
 \Delta f^{n+1}=f^{n+1}-f^n,\qquad
 f\in\{\rho,\psi,u,h,k\},
\]
and define
\begin{align}
 \mathcal E_{n+1}(t)={}&
 \norm[H^1]{\Delta\rho^{n+1}}^2
 +\norm[L^2]{\Delta\psi^{n+1}}^2
 +\norm[H^1]{\Delta u^{n+1}}^2+\norm[H^1]{\Delta h^{n+1}}^2
 +\norm[L^2]{\Delta k^{n+1}}^2,
 \quad
 \mathcal D_{n+1}=\norm[H^2]{\Delta u^{n+1}}^2.
 \label{eq:app-picard-low-energy}
\end{align}
On the ball determined by $R_*$, the Lorentz map is locally Lipschitz:
\begin{align}
 \norm[L^2]{\mathcal M(U_1)-\mathcal M(U_2)}
 \leq C_{\eta,R_*}\bigl(&\norm[H^1]{\rho_1-\rho_2}
 +\norm[L^2]{\psi_1-\psi_2}
 +\norm[H^1]{h_1-h_2}
 +\norm[L^2]{k_1-k_2}\bigr).
 \label{eq:app-M-lipschitz-picard}
\end{align}
This follows by subtracting the two products in
\eqref{eq:app-M-iterate}; for example,
\[
 \rho_1^2h_1\cdot k_1-\rho_2^2h_2\cdot k_2
 =(\rho_1^2-\rho_2^2)h_1\cdot k_1
 +\rho_2^2(h_1-h_2)\cdot k_1
 +\rho_2^2h_2\cdot(k_1-k_2).
\]
Subtract the transport equations at two successive iterates, multiply by
the corresponding differences, and integrate.  Subtracting the parabolic
equations and testing by $\Delta u^{n+1}$ and
$-\Delta u_{xx}^{n+1}$ gives, after Young's inequality,
\begin{equation}\label{eq:app-picard-contraction}
 \sup_{0\leq s\leq t}\mathcal E_{n+1}(s)
 +\int_0^t\mathcal D_{n+1}(s)\,\dd s
 \leq C_{\eta,R_*}t
 \sup_{0\leq s\leq t}\mathcal E_n(s).
\end{equation}
The initial differences vanish.  All convection terms in this calculation
have the form ``bounded coefficient times one old and one new difference'';
\eqref{eq:app-M-lipschitz-picard} treats the only new magnetic term.
Choose $T_\eta>0$ so that $C_{\eta,R_*}T_\eta\leq1/2$.  Then
\eqref{eq:app-picard-contraction} makes the iterates Cauchy in the lower
norm and constructs a positive-background solution on $[0,T_\eta]$.

Estimate \eqref{eq:app-uniform-bound} and the positive-density continuation
criterion extend it to the common interval $[0,T_*]$.  Indeed, the
characteristic formula gives
\begin{equation}\label{eq:app-fixed-eta-lower}
 \inf_{x\in\R}\rho^\eta(t,x)
 \ge r_\eta\exp\left(
      -\int_0^t\norm[L^\infty]{u_x^\eta(s)}\,\dd s\right)>0,
 \qquad 0\le t\le T_*.
\end{equation}
The integral in \eqref{eq:app-fixed-eta-lower} is bounded uniformly by
\eqref{eq:app-uniform-bound}.  The lower bound may depend on $\eta$, which
is sufficient for continuation of each fixed positive-background problem.
Thus no
$\eta$-uniform contraction constant is required: contraction constructs the
solution for each fixed $\eta$, while the uniform energy estimate supplies
the common lifespan.

\subsection{Time compactness and passage to far-field vacuum}

Let $K\Subset\R$.  The equations and \eqref{eq:app-uniform-bound} imply,
uniformly in $\eta$,
\begin{align}
 \phi_t^\eta,\rho_t^\eta
 &\quad\text{bounded in }L^2(0,T_*;H^1(K)),\nonumber\\
 \psi_t^\eta
 &\quad\text{bounded in }L^2(0,T_*;L^2(K)),\nonumber\\
 u_t^\eta
 &\quad\text{bounded in }
 L^\infty(0,T_*;L^2)\cap L^2(0,T_*;H^1),\nonumber\\
 h_t^\eta
 &\quad\text{bounded in }L^2(0,T_*;H^2(K)),\nonumber\\
 k_t^\eta
 &\quad\text{bounded in }L^2(0,T_*;H^1(K)).
 \label{eq:app-time-compactness}
\end{align}
Aubin--Lions and the one-dimensional embeddings therefore give, after
taking a subsequence,
\begin{align}
 \rho^\eta&\longrightarrow\rho
     &&\text{strongly in }C([0,T_*];H^1(K)),\nonumber\\
 u^\eta&\longrightarrow u
     &&\text{strongly in }C([0,T_*];H^1(K)),\nonumber\\
 \psi^\eta&\longrightarrow\psi
     &&\text{strongly in }L^2(0,T_*;L^2(K)),\nonumber\\
 h^\eta&\longrightarrow h
     &&\text{strongly in }C([0,T_*];H^2(K)),\nonumber\\
 k^\eta&\longrightarrow k
     &&\text{strongly in }C([0,T_*];H^1(K)).
 \label{eq:app-strong-convergence}
\end{align}
In particular,
\begin{equation}\label{eq:app-M-local-convergence}
 \mathcal M^\eta\longrightarrow
 |h|^2\rho_x+\rho^2h\cdot k
 \quad\text{strongly in }
 L^2(0,T_*;L^2(K)).
\end{equation}
This local convergence is sufficient to pass to the equations in
distributions.  The singular viscous terms pass to the limit exactly as in
\cite[Section 2]{CaoLiZhu2022}; the magnetic term passes by
\eqref{eq:app-M-local-convergence}.  The nonsingular constraint
$h_x^\eta=\rho^\eta k^\eta$ passes directly.  We do not pass to the limit
in $Q^\eta$, which contains a negative density power.  Instead,
\eqref{eq:app-Q-iterate} gives the equivalent identity
\begin{equation}\label{eq:app-nonsingular-density-limit}
 (\rho^\eta)_x
 =\frac1\delta(\rho^\eta)^{2-\delta}\psi^\eta
 \longrightarrow
 \frac1\delta\rho^{2-\delta}\psi
 \quad\text{in }L^2(0,T_*;L^2(K)).
\end{equation}
Thus, in the distributional sense,
\[
 h_x=\rho k,\qquad
 \rho_x=\frac1\delta\rho^{2-\delta}\psi.
\]
At every finite point $\rho>0$, the second identity is equivalent to
$\psi=\delta\rho^{\delta-2}\rho_x$.
Therefore $b=\rho h$ satisfies $b_t+(ub)_x=0$.
The characteristic formula for $\phi$ preserves positivity at every
finite point, and the cutoff argument of
\cite[Lemma 2.2]{CaoLiZhu2022} yields
$\rho\in C([0,T_*];L^1)$.

\subsection{Difference estimates and uniqueness}

Let $(\rho_i,u_i,h_i,k_i)$, $i=1,2$, be two regular solutions with the
same initial data and the same two-tail profile $\bar h$.  Set
\[
 r=\rho_1-\rho_2,\qquad
 w=u_1-u_2,\qquad
 z=h_1-h_2,\qquad
 \ell=k_1-k_2.
\]
Put
\begin{equation}\label{eq:app-fluid-difference-variables}
 \vartheta_i=\rho_i^{1-\delta},\qquad
 f_i=(\log\rho_i)_x,\qquad
 \Phi=\phi_1-\phi_2,\quad
 \Theta=\vartheta_1-\vartheta_2,\quad
 F=f_1-f_2.
\end{equation}
Thus $\vartheta_i\equiv1$ and $\Theta\equiv0$ when $\delta=1$.
The densities of two regular solutions with the same initial density are
comparable:
\begin{equation}\label{eq:app-density-comparability}
 C_T^{-1}\le\frac{\rho_1(t,x)}{\rho_2(t,x)}\le C_T,
 \qquad (t,x)\in[0,T]\times\R.
\end{equation}
Indeed, if $A_i(t,x)$ is the inverse flow label for $u_i$, then
\[
 \rho_i(t,x)=\rho_0(A_i(t,x))
 \exp\left(-\int_0^t u_{ix}(s,X_i(s;A_i(t,x)))\,\dd s\right).
\]
The displacements are bounded according to
\[
 |A_i(t,x)-x|\le\int_0^T\|u_i(s)\|_\infty\,\dd s,
\]
while
$(\log\rho_0)_x=\delta^{-1}\rho_0^{1-\delta}\psi_0\in L^\infty$
(with the evident interpretation for $\delta=1$).  This proves
\eqref{eq:app-density-comparability}.

Multiplying the velocity equation by $\vartheta_i$ gives the exact
nonsingular principal equation
\begin{equation}\label{eq:app-transformed-momentum}
 \vartheta_i(u_{it}+u_i u_{ix}+\phi_{ix})-\alpha u_{ixx}
 =\alpha\delta f_i u_{ix}-\vartheta_i\mathcal M_i.
\end{equation}
Subtracting the transport equations gives the following explicit system for
the fluid differences:
\begin{align}
 \Phi_t+u_1\Phi_x+(\gamma-1)\phi_1w_x
 &=-w\phi_{2x}-(\gamma-1)\Phi u_{2x},
 \label{eq:app-Phi-difference}\\
 \Theta_t+u_1\Theta_x+(1-\delta)\vartheta_1w_x
 &=-w\vartheta_{2x}-(1-\delta)\Theta u_{2x},
 \label{eq:app-Theta-difference}\\
 F_t+u_1F_x+u_{1x}F+w_{xx}
 &=-wf_{2x}-w_xf_2.
 \label{eq:app-F-difference}
\end{align}
The second equation is absent when $\delta=1$, because then
$\Theta\equiv0$.  Subtracting \eqref{eq:app-transformed-momentum} gives
\begin{align}
 \vartheta_1(w_t+u_1w_x+\Phi_x)-\alpha w_{xx}
 ={}&-\vartheta_1wu_{2x}
 -\Theta(u_{2t}+u_2u_{2x}+\phi_{2x})\nonumber\\
 &+\alpha\delta f_1w_x+\alpha\delta Fu_{2x}-\mathcal N,
 \qquad
 \mathcal N=\vartheta_1\mathcal M_1-\vartheta_2\mathcal M_2.
 \label{eq:app-momentum-difference-explicit}
\end{align}
These formulas can be checked by adding and subtracting the terms with
subscript two.  They also show why multiplication by $\vartheta_i$ is
useful: the coefficient of $w_{xx}$ is the positive constant $\alpha$.

Define the explicit one-dimensional fluid difference functional and its
dissipation by
\begin{align}
 \mathfrak E_f={}&
  \norm[H^1]{\Phi}^2+\norm[H^1]{\Theta}^2+\norm[L^2]{F}^2
  +\norm[L^2]{\sqrt{\vartheta_1}w}^2+\norm[L^2]{w_x}^2,
 \label{eq:app-fluid-functional}\\
 \mathfrak D_f={}&
  \norm[L^2]{w_x}^2+\norm[L^2]{\sqrt{\vartheta_1}w_t}^2.
 \label{eq:app-fluid-dissipation}
\end{align}
These are the one-dimensional quantities used in the uniqueness calculation
of \cite[Section 3.5]{XinZhu2021}.  We now verify the additional magnetic
pairings in that calculation.
If $0<\delta<1$, the identities
\begin{equation}\label{eq:app-weighted-force-difference}
 \vartheta_i\mathcal M_i
 =|h_i|^2\rho_i^{2-\delta}f_i
   +\rho_i^{3-\delta}h_i\cdot k_i
\end{equation}
give the exact decomposition
\begin{align}
 \mathcal N={}&
 (|h_1|^2-|h_2|^2)\rho_1^{2-\delta}f_1
 +|h_2|^2(\rho_1^{2-\delta}-\rho_2^{2-\delta})f_1
 +|h_2|^2\rho_2^{2-\delta}F\nonumber\\
 &+(\rho_1^{3-\delta}-\rho_2^{3-\delta})h_1\cdot k_1
 +\rho_2^{3-\delta}z\cdot k_1
 +\rho_2^{3-\delta}h_2\cdot\ell.
 \label{eq:app-N-expansion-sublinear}
\end{align}
The fields $h_i$ themselves may have nonzero tail values and need not lie
in $L^2$.  In \eqref{eq:app-N-expansion-sublinear} they are used only in
$L^\infty$; the difference $z=h_1-h_2$ lies in $H^1$ because the two
solutions have the same tail profile.  Formula
\eqref{eq:app-N-expansion-sublinear}, together with
\eqref{eq:app-density-comparability}, implies
\begin{equation}\label{eq:app-N-bound-sublinear}
 \norm[L^2]{\mathcal N/\sqrt{\vartheta_1}}
 \le C_T\left(
  \norm[H^1]{\Theta}+\norm[L^2]{F}
  +\norm[H^1]{z}+\norm[L^2]{\ell}\right).
\end{equation}
To see this without a hidden vacuum singularity, write
$\rho^q=\vartheta^{q/(1-\delta)}$ in
\eqref{eq:app-weighted-force-difference}.  Both exponents
$(2-\delta)/(1-\delta)$ and $(3-\delta)/(1-\delta)$ are larger than
$3/2$; the mean-value theorem, followed by division by
$\sqrt{\vartheta_1}$, therefore leaves only nonnegative powers of
$\vartheta_i$.  The remaining factors are controlled by the one-dimensional
Sobolev embedding.

When $\delta=1$, $\mathcal N=\mathcal M_1-\mathcal M_2$.  The force
representation gives
\begin{align}
 \mathcal M_1-\mathcal M_2
={}&(|h_1|^2-|h_2|^2)\rho_{1x}+|h_2|^2r_x +(\rho_1^2-\rho_2^2)h_1\cdot k_1
   +\rho_2^2z\cdot k_1+\rho_2^2h_2\cdot\ell,
 \label{eq:app-M-difference}
\end{align}
and hence
\begin{equation}\label{eq:app-N-bound-linear}
 \norm[L^2]{\mathcal N}
 \le C_T\left(
   \norm[H^1]{r}+\norm[H^1]{z}+\norm[L^2]{\ell}\right).
\end{equation}
Set
\begin{equation}\label{eq:app-magnetic-difference-remainder}
 \mathfrak R_m=
 \begin{cases}
  \norm[H^1]{z}^2+\norm[L^2]{\ell}^2,&0<\delta<1,\\
  \norm[H^1]{r}^2+\norm[H^1]{z}^2+\norm[L^2]{\ell}^2,&\delta=1.
 \end{cases}
\end{equation}
Let $c_0>0$ denote the coercivity constant in the hydrodynamic difference
estimate.
The only new terms in the zeroth- and first-order velocity-difference
identities are now explicit:
\begin{align}
 |\ip{\mathcal N}{w}|
 &\le C_T(\mathfrak E_f+\mathfrak R_m),
 \label{eq:app-Nw-pairing}\\
 |\ip{\mathcal N}{w_t}|
 &\le \frac{c_0}{4}
       \norm[L^2]{\sqrt{\vartheta_1}w_t}^2
       +C_T(\mathfrak E_f+\mathfrak R_m).
 \label{eq:app-Nwt-pairing}
\end{align}
Here \eqref{eq:app-N-bound-sublinear} is used for $0<\delta<1$ and
\eqref{eq:app-N-bound-linear} for $\delta=1$.  Applying the hydrodynamic
difference calculation can be followed in four elementary tests:
\begin{enumerate}[label=\textup{(\roman*)},leftmargin=2.2em]
\item Multiply \eqref{eq:app-momentum-difference-explicit} by $w$.  After
integrating $-\alpha w_{xx}w$, the positive term is
$\alpha\norm[L^2]{w_x}^2$.  Derivatives of $\vartheta_1$ are controlled by
its transport equation.
\item Multiply the same equation by $w_t$.  This produces
$\norm[L^2]{\sqrt{\vartheta_1}w_t}^2$ and
$\frac\alpha2\frac{\dd}{\dd t}\norm[L^2]{w_x}^2$.
The two Lorentz pairings are exactly
\eqref{eq:app-Nw-pairing}--\eqref{eq:app-Nwt-pairing}.
\item Multiply \eqref{eq:app-Phi-difference} and
\eqref{eq:app-Theta-difference}, and their first $x$ derivatives, by the
corresponding differences.  Multiply \eqref{eq:app-F-difference} by $F$.
All transport terms are integrated by parts; for example,
$\int u_1\Phi_x\Phi=-\frac12\int u_{1x}\Phi^2$.
\item Use Young's inequality to assign a small part of every $w_x$, $w_t$,
or $w_{xx}$ product to the corresponding dissipation.  The remaining
coefficients are bounded by the regular-solution norms on $[0,T]$.
\end{enumerate}
Adding these tests and inserting
\eqref{eq:app-Nw-pairing}--\eqref{eq:app-Nwt-pairing} gives
\begin{equation}\label{eq:app-fluid-difference}
 \frac{\dd}{\dd t}\mathfrak E_f+c_0\mathfrak D_f
 \le C_T\bigl(\mathfrak E_f+\mathfrak R_m\bigr).
\end{equation}
No spatial derivative of $\mathcal N$ is used in this difference estimate.
Subtracting \eqref{eq:app-transformed-momentum} also gives the elliptic
recovery
\begin{equation}\label{eq:app-fluid-elliptic-recovery}
 \norm[L^2]{w_{xx}}^2
 \le C_T\bigl(\mathfrak E_f+\mathfrak D_f+\mathfrak R_m\bigr).
\end{equation}

It remains to estimate the transported magnetic differences.  They satisfy
\begin{align}
 z_t+u_1z_x&=-w(h_2)_x,\label{eq:app-z-difference}\\
 \ell_t+u_1\ell_x&=-w(k_2)_x.\label{eq:app-l-difference}
\end{align}
The support of $k_i(t)$ stays in a fixed compact interval $K_T$.
Consequently $(h_i)_x=\rho_i k_i$ is supported in $K_T$, and
$\vartheta_1$ has a positive lower bound there.  In particular,
\begin{equation}\label{eq:app-local-w-control}
 \norm[H^1(K_T)]{w}^2
 \le C_T\left(
   \norm[L^2]{\sqrt{\vartheta_1}w}^2+\norm[L^2]{w_x}^2\right)
 \le C_T\mathfrak E_f.
\end{equation}
The $H^1$--$L^2$
transport estimates therefore yield
\begin{equation}\label{eq:app-magnetic-transport-difference}
 \frac{\dd}{\dd t}\left(\norm[H^1]{z}^2+\norm[L^2]{\ell}^2\right)
 \le C_T\left(\norm[H^1]{z}^2+\norm[L^2]{\ell}^2
                         +\mathfrak E_f\right).
\end{equation}
For $0<\delta<1$, combining
\eqref{eq:app-fluid-difference} and
\eqref{eq:app-magnetic-transport-difference} closes the argument directly.

For $\delta=1$ one also propagates the raw density difference.  The
continuity equations imply
\begin{equation}\label{eq:app-r-difference}
 r_t+u_1r_x+\rho_1w_x=-w\rho_{2x}-ru_{2x}.
\end{equation}
Since $\vartheta_1=1$, $\mathfrak E_f$ controls the full $H^1$ norm of
$w$.  Differentiating \eqref{eq:app-r-difference} once gives, for every
$\varepsilon>0$,
\begin{equation}\label{eq:app-r-energy}
 \frac{\dd}{\dd t}\norm[H^1]{r}^2
 \le C_{\varepsilon,T}\left(\norm[H^1]{r}^2+\mathfrak E_f\right)
      +\varepsilon\norm[L^2]{w_{xx}}^2.
\end{equation}
Equations \eqref{eq:app-r-energy} and
\eqref{eq:app-fluid-elliptic-recovery} close the $\delta=1$ case.

Choose $\lambda>0$ sufficiently small and define
\begin{equation}\label{eq:app-total-difference}
 \mathfrak E=\mathfrak E_f+\lambda\mathfrak R_m.
\end{equation}
Adding the corresponding transport estimates to
\eqref{eq:app-fluid-difference}, choosing $\varepsilon$ and then $\lambda$
small enough, and using \eqref{eq:app-fluid-elliptic-recovery} when
$\delta=1$, gives
\begin{equation}\label{eq:app-final-uniqueness}
 \frac{\dd}{\dd t}\mathfrak E
 +\frac{c_0}{2}\mathfrak D_f
 \le C_T\mathfrak E.
\end{equation}
Since $\mathfrak E(0)=0$, Gronwall's inequality proves uniqueness.  In the
sublinear case $\Theta=0$ implies $\rho_1=\rho_2$ pointwise; in the linear
case this follows from $r=0$.

Finally, the bounds in the continuation quantity give weakly continuous
terminal traces in all top-order spaces and strong traces one derivative
below.  Equations
\eqref{eq:app-rhox}--\eqref{eq:app-rhoxx} preserve the density data class,
the weighted velocity bounds preserve the fluid compatibility conditions,
and the characteristic formulas preserve
$k\in H^2$, compact support of $k$, and
$h\in\mathcal H_{\bar h}^3$.  The construction above can therefore be
restarted at the terminal time, which proves the continuation assertion in
\cref{prop:local}.

\section*{Data availability}
Data sharing is not applicable to this article because no datasets were
generated or analyzed in the course of this study.

\section*{Declaration of competing interest}
The authors declare that they have no known competing financial interests or
personal relationships that could have appeared to influence the work
reported in this paper.

\section*{CRediT authorship contribution statement}
Omitted from this anonymous version and to be supplied in the non-anonymous
manuscript.

\section*{Acknowledgments}
Omitted for anonymous review.

\bibliographystyle{unsrtnat}
\bibliography{references}

@article{CaoLiZhu2022,
  author  = {Cao, Yue and Li, Hao and Zhu, Shengguo},
  title   = {Global Regular Solutions for One-Dimensional Degenerate
             Compressible {Navier--Stokes} Equations with Large Data
             and Far-Field Vacuum},
  journal = {SIAM Journal on Mathematical Analysis},
  year    = {2022},
  volume  = {54},
  number  = {4},
  pages   = {4658--4694},
  doi     = {10.1137/21M1464609}
}

@article{Yu2013,
  author  = {Yu, Haibo},
  title   = {Global Classical Large Solutions with Vacuum to {1D}
             Compressible {MHD} with Zero Resistivity},
  journal = {Acta Applicandae Mathematicae},
  year    = {2013},
  volume  = {128},
  number  = {1},
  pages   = {193--209},
  doi     = {10.1007/s10440-013-9826-3}
}

@article{YeLi2019,
  author  = {Ye, Yulin and Li, Zilai},
  title   = {Global Strong Solution to the {Cauchy} Problem of {1D}
             Compressible {MHD} Equations with Large Initial Data
             and Vacuum},
  journal = {Zeitschrift f{\"u}r angewandte Mathematik und Physik},
  year    = {2019},
  volume  = {70},
  number  = {1},
  pages   = {38},
  note    = {Article 38},
  doi     = {10.1007/s00033-019-1078-2}
}

@article{LiWangYe2020,
  author  = {Li, Zilai and Wang, Huaqiao and Ye, Yulin},
  title   = {Global Strong Solutions to the {Cauchy} Problem of {1D}
             Compressible {MHD} Equations with No Resistivity},
  journal = {Communications in Mathematical Sciences},
  year    = {2020},
  volume  = {18},
  number  = {3},
  pages   = {851--873},
  doi     = {10.4310/CMS.2020.v18.n3.a12}
}

@article{GuanWei2025,
  author  = {Guan, Junhua and Wei, Suting},
  title   = {Global Classical Solutions to the {1D} Compressible
             Non-Resistive Magnetohydrodynamic Equations with
             Density-Dependent Viscosity},
  journal = {Journal of Mathematical Physics},
  year    = {2025},
  volume  = {66},
  number  = {12},
  pages   = {121514},
  doi     = {10.1063/5.0294874}
}

@misc{LiuLuoZhong2025,
  author        = {Liu, Hairong and Luo, Tao and Zhong, Hua},
  title         = {Strong Solutions to the {3-D} Compressible {MHD}
                   Equations with Density-Dependent Viscosities in
                   Exterior Domains with Far-Field Vacuum},
  year          = {2025},
  eprint        = {2504.04376},
  archivePrefix = {arXiv},
  primaryClass  = {math.AP},
  howpublished  = {arXiv:2504.04376},
  doi           = {10.48550/arXiv.2504.04376}
}

@misc{LiLuYuan2026,
  author        = {Li, Jiaxu and L{\"u}, Boqiang and Yuan, Bing},
  title         = {Strong Solutions to the Initial-Boundary-Value Problem
                   of Compressible {MHD} Equations with Degenerate
                   Viscosities and Far Field Vacuum in {3D} Exterior
                   Domains},
  year          = {2026},
  eprint        = {2602.05264},
  archivePrefix = {arXiv},
  primaryClass  = {math.AP},
  howpublished  = {arXiv:2602.05264},
  doi           = {10.48550/arXiv.2602.05264}
}

@article{BreschDesjardins2003,
  author  = {Bresch, Didier and Desjardins, Beno{\^i}t},
  title   = {Existence of Global Weak Solutions for a {2D} Viscous
             Shallow Water Equations and Convergence to the
             Quasi-Geostrophic Model},
  journal = {Communications in Mathematical Physics},
  year    = {2003},
  volume  = {238},
  number  = {1--2},
  pages   = {211--223},
  doi     = {10.1007/s00220-003-0859-8}
}

@article{BreschDesjardins2007,
  author  = {Bresch, Didier and Desjardins, Beno{\^i}t},
  title   = {On the Existence of Global Weak Solutions to the
             {Navier--Stokes} Equations for Viscous Compressible
             and Heat Conducting Fluids},
  journal = {Journal de Math{\'e}matiques Pures et Appliqu{\'e}es},
  year    = {2007},
  volume  = {87},
  number  = {1},
  pages   = {57--90},
  doi     = {10.1016/j.matpur.2006.11.001}
}

@article{LiPanZhu2017,
  author  = {Li, Yue and Pan, Ronghua and Zhu, Shengguo},
  title   = {On Classical Solutions to {2D} Shallow Water Equations with
             Degenerate Viscosities},
  journal = {Journal of Mathematical Fluid Mechanics},
  year    = {2017},
  volume  = {19},
  pages   = {151--190},
  doi     = {10.1007/s00021-016-0276-3}
}

@article{XinZhu2021,
  author  = {Xin, Zhouping and Zhu, Shengguo},
  title   = {Well-Posedness of Three-Dimensional Isentropic Compressible
             {Navier--Stokes} Equations with Degenerate Viscosities and
             Far Field Vacuum},
  journal = {Journal de Math{\'e}matiques Pures et Appliqu{\'e}es},
  year    = {2021},
  volume  = {152},
  pages   = {94--144},
  doi     = {10.1016/j.matpur.2021.05.004}
}

\end{document}